\documentclass[final,3p,times]{elsarticle}

\usepackage{amssymb,amsmath,amsfonts,mathtools}
\usepackage{amsthm}

\newtheorem*{lemma}{Lemma}

\usepackage{graphicx}
\usepackage{graphbox}
\usepackage{algorithm}
\usepackage{algpseudocode}
\usepackage{subcaption}
\usepackage{tabularx,booktabs}
\newcolumntype{R}{>{\raggedright\arraybackslash}X}
\newcolumntype{L}{>{\raggedleft\arraybackslash}X}
\newcolumntype{Y}{>{\centering\arraybackslash}X}

\usepackage{multirow}
\usepackage{systeme}
\biboptions{sort&compress}
\usepackage{lineno}
\modulolinenumbers[5]
\usepackage[]{setspace}

\usepackage{listings}
\usepackage{color} 
\definecolor{mygreen}{RGB}{28,172,0} 
\definecolor{mylilas}{RGB}{170,55,241}

\journal{Computer Methods in Applied Mechanics and Engineering}

\begin{document}

\begin{frontmatter}

\title{Smooth Implicit Hybrid Upwinding \\ for Compositional Multiphase Flow in Porous Media}


\author[St]{Sebastian BM Bosma\corref{cor1}}
\ead{sbosma@stanford.edu}

\author[Total]{Fran\c cois P Hamon}
\ead{francois.hamon@totalenergies.com}

\author[Chevron]{Brad T Mallison}
\ead{btmb@chevron.com}

\author[St]{Hamdi A Tchelepi}
\ead{tchelepi@stanford.edu}

\cortext[cor1]{Corresponding authors}

\address[St]{Department of Energy Resources Engineering, Stanford University, Stanford, CA, USA}
\address[Total]{TotalEnergies E\&P Research and Technology, Houston, TX 77002, USA}
\address[Chevron]{Chevron Technical Center, Houston, TX 77002, USA}

\begin{abstract}
In subsurface multiphase flow simulations, poor nonlinear solver performance is a significant runtime sink. The system of fully implicit mass balance equations is highly nonlinear and often difficult to solve for the nonlinear solver, generally Newton(-Raphson). Although the physical problem is inherently nonlinear, discretization schemes introduce several strong nonlinearities that can cause Newton iterations to diverge or converge very slowly. This frequently results in time step cuts, leading to wasted iterations and computationally expensive simulations. Much literature has looked into how to improve the nonlinear solver through enhancements or safeguarding updates. In this work, we take a different approach; we aim to improve convergence with a smoother finite volume discretization scheme which is more suitable for the Newton solver.

Building on recent work, we propose a novel total velocity hybrid upwinding scheme with weighted average flow mobilities (WA-HU TV) that is unconditionally monotone and extends to compositional multiphase simulations. Analyzing the solution space of a one-cell problem, we demonstrate the improved properties of the scheme and explain how it leverages the advantages of both phase potential upwinding and arithmetic averaging. This results in a flow subproblem that is smooth with respect to changes in the sign of phase fluxes, and is well-behaved when phase velocities are large or when co-current viscous forces dominate. Additionally, we propose a WA-HU scheme with a total mass (WA-HU TM) formulation that includes phase densities in the weighted averaging. 

The proposed WA-HU TV consistently outperforms existing schemes, yielding benefits from 5\% to over 50\% reduction in nonlinear iterations. The WA-HU TM scheme also shows promising results; in some cases leading to even more efficiency. However, WA-HU TM can occasionally also lead to convergence issues. Overall, based on the current results, we recommend the adoption of the WA-HU TV scheme as it is highly efficient and robust.

\end{abstract}

\begin{keyword}
Hybrid Upwinding \sep
Efficient Nonlinear Solver \sep
Coupled Flow and Transport \sep
Geologic Carbon Sequestration \sep
Reservoir Simulation \sep
Implicit Discretization Scheme
\end{keyword}

\end{frontmatter}



\allowdisplaybreaks

\section{Introduction}
As the world attempts to mitigate global warming and pursue less than $2.5^{\circ}$C degree pathways \cite{IPCC18}, reducing carbon emissions has become an immediate priority. Cleaner gas-powered plants are replacing coal at a rapid rate, and numerous CCS (Carbon Capture \& Sequestration) projects are being drawn up as carbon policies are put into place. With this, the demand for efficient simulators, which can accurately predict the migration of gas through liquid-filled subsurface formations, has soared. 
Strong buoyancy and capillary forces present in these gas-liquid processes make the partial differential equations governing coupled multiphase flow and transport inherently hard to solve. In combination with the strong heterogeneities of reservoir rock properties, often spanning multiple orders of magnitude, carbon sequestration and gas production simulations present a significant challenge. As such, industry-standard simulation techniques have started to fall short of the desired efficiency to perform the numerous full-scale simulations required for project planning and risk assessments.

A significant runtime sink is poor nonlinear solver performance. The nonlinear solver, most commonly Newton's method with damping \cite{DeuflhardBook,Eclipse}, is used to solve the algebraic system of equations arising from the implicit time discretization of the governing reservoir simulation equations. The system can be described as a near elliptic subproblem for flow and a near hyperbolic subproblem for transport \cite{Trangenstein1989MathematicalStructure}, that have strong nonlinear coupling through fluid and rock properties. To determine numerical phase fluxes in the finite volume method, the industry-standard discretization scheme uses a conventional formulation with phase potential upwinding (PPU). By doing so, the coupled flow and transport problem relies on a single upwinding decision per phase based on the sign of the potential difference. Importantly, these upwinding decisions result in kinks in the solution space, i.e, discontinuities in the derivative, that can impact Newton's ability to converge. Traditional improvement techniques, such as damping, are based on heuristics and only have limited success. More recently, trust-region chopping strategies \cite{WANG2013TrustRegion,Li2015TrustRegion,Moyner2017TrustRegions,Klemetsdal2019TR} have offered a robust alternative, yet add computational cost per iteration. 

In the development of such Newton improvement strategies, it has become evident that strong nonlinearities in the numerical solution space cause slow convergence or divergence of the Newton iterations \cite{Jenny2009, WANG2013TrustRegion, Hamon16HU_Analysis, Voskov2012EoSsimulation}. Although some challenging nonlinear features are inherent to the physics -- such as inflection points of fractional flow curves -- it is desirable to avoid kinks or strong curvature induced by the numerical treatment of governing equations. Moreover, many works have shown that smoother discretizations lead to improvements in efficiency \cite{Hamon2016HU_Buoyancy,JIANG2017C1PPU, Khebzegga2020Continuous, Jiang2021SmoothCompositional}. 
To this end, a total velocity formulation with hybrid upwinding (HU) has been proposed. 
Originally developed to improve the properties of a transport problem--that is, considering a fixed total velocity--hybrid upwinding applies a separate treatment to the fluid properties appearing in the viscous, gravitational, and capillary terms of the numerical flux.
As a result, where the standard PPU scheme sees severe kinks related to changes in the single upwinding direction per phase, this hybrid approach yields several softer kinks related to the upwinding choices per force.
This smoother HU-based discretization leads to a reduction in the number of Newton iterations for pure (immiscible) transport problems as documented in previous work  \cite{LEE2015HU,Hamon16HU_Analysis}.

However, in implicit methods for coupled flow and transport--that we exclusively consider in this work--the total velocity is a space- and time-dependent function of the primary variables.
Moreover, numerical studies \cite{Hamon2016HU_Buoyancy,BRENNER2020} have shown that the discretization of the mobilities appearing in the total velocity has a strong impact on Newton convergence as these fluid properties dictate the flow problem and are key in the coupling between flow and transport.
More concretely, as the total velocity in the original HU scheme is computed using PPU, the troublesome kinks related to changes in sign of the phase potential are still present in the coupled flow and transport problem. 
To address this issue, Hamon et al. \cite{Hamon2016HU_Buoyancy} employed a total velocity formulation combining HU with a smooth weighted averaging (WA) of the mobilities used to compute the total velocity. 
Their work demonstrates that this WA-HU approach produces excellent nonlinear behavior in the context of immiscible, incompressible two-phase flow and transport with buoyancy.
This scheme was then extended to three-phase flow and transport with buoyancy and capillarity where capillary forces are treated outside of the total velocity \cite{Hamon2018HU_capillarity}. 
In both of these schemes a single weighting coefficient is applied to all phases. 
As a result, the parameters of the weighting coefficient are bounded to prevent downwinding, and the extension of the scheme to compositional fluids is non-trivial.


We propose a generalized hybrid upwinding scheme with improved weighted averaging of the mobilities used to compute the total velocity (WA-HU). The discretization is unconditionally monotone and extends to compositional multiphase simulations. (i) With a separate weighting coefficient for each phase, the weighted mobilities are designed to yield robust nonlinear behavior in strong co-current flow regimes, as well as in counter-current flow regimes where phase velocities may change sign frequently. The separate treatment per phase also removes the lower bound on the weighting coefficient and thus resolves issues mentioned in \cite{JIANG2017C1PPU} where the original WA-HU could lead to noticeable differences with respect to PPU. Secondly, (ii) we extend the original hybrid upwinding transport framework such that it is well-suited for more than two phases. These two contributions, (i) and (ii), result in a scheme with a coupling of the flow and transport which reduces the number of iterations required by the nonlinear solver while preserving the accuracy of the discretization. Furthermore, the flexibility of the proposed WA-HU also enables smooth upwinding of properties such as density. Hence, as a third contribution, (iii) we present and motivate a novel phase flux computation using a total mass flux formulation. We ultimately evaluate WA-HU's performance, both in a total velocity and total mass flux scheme, with respect to existing schemes by comparing cumulative Newton iteration counts. 

Section \ref{sec:GovEq} presents the governing equations. We then describe the numerical discretization and flux formulations in section \ref{sec:NumDisc}. Next, treatment of the numerical fluid properties is discussed in section \ref{sec:FluidPropertyTreatment}. We perform a quantitative analysis of the existing and proposed schemes in section \ref{sec:Analysis}. Finally, the schemes are tested on a suite of numerical test cases in section \ref{sec:TestCases}. Section \ref{sec:conclusion} summarizes the conclusions and offers some points of discussion.

\section{Compositional Flow and Transport through Porous Media}  \label{sec:GovEq}
We consider an incompressible porous medium filled with an isothermal reservoir fluid system with $n_c$ components and $n_p$ phases. 
To describe this system, we implement a general compositional model. This choice is made for generality; the presented work is also applicable to formulations such as the Black-Oil model. The compositional mass conservation equations are expressed in their differential form as
\begin{align}
    \phi \ \frac{\partial}{\partial t} \Bigg( \sum_\ell^{n_p} x_{c,\ell} \rho_\ell S_\ell \Bigg) + \nabla \cdot \sum_\ell^{n_p}  x_{c,\ell} \rho_\ell \boldsymbol{u}_\ell = \sum_\ell^{n_p}  x_{c,\ell} \rho_\ell \Tilde{q}^s_{\ell} 
    \quad \quad \forall \ c \in \{1,...,n_c\},
    \label{eq:DifferentialFormMassConservation}
\end{align}
where $\phi$ is the porosity, $\rho_\ell$ is the phase mass density, $x_{c,\ell}$ is the mass fraction of component $c$ in phase $\ell$, and $S_\ell$ is the phase saturation variable. Sources and sinks, commonly wells in reservoir simulation, are captured by $\Tilde{q}_\ell^s$. Here, $\Tilde{\bullet}$ designates a per-unit-volume flow quantity. The phase velocity, $\boldsymbol{u}_\ell$, is characterized by Darcy's law:
\begin{align}
    \boldsymbol{u}_\ell = - k \lambda_\ell \nabla \Phi_\ell 
    \quad \forall \ \ell \in \{1,...,n_p\}, \label{eq:Darcy}
\end{align}
where $k$ is the scalar permeability of the porous medium and $\lambda_\ell = \frac{k_{r,\ell}}{\mu_\ell}$ is the fluid mobility, equal to the relative permeability over the viscosity. We assume that $k_{r,\ell}$ is a non-decreasing function of its own phase saturation; in this work we use Brooks-Corey relationships. The phase potential gradient is defined as
\begin{align}
     \nabla \Phi_\ell = \nabla p_\ell - \rho_\ell g \nabla z
\end{align}
Here, $p_\ell$ is the phase pressure variable, $g$ is the gravitational acceleration ($9.81 \text{m/s}^2$) and $z$ is depth (defined positive downwards). In addition to the mass balance equations, we impose: the saturation constraint $\sum_\ell S_\ell = 1$, $n_p$ mass fraction constraints expressed as
\begin{align}
    \sum_c^{n_c} x_{c,\ell} = 1  \quad \forall \ \ell \in \{1,...,n_p\}, \label{eq:compFracConstraint}
\end{align}
and $n_p$ capillary pressure relationships which relate the phase pressure to a reference pressure $p$, i.e. 
\begin{align}
    p_{cap,\ell}(S_\ell) = p - p_\ell, \quad \ell \in \{1,...,n_p\}. 
\end{align}
Typically the reference pressure is chosen to be one of the phase pressures. Hence the capillary function of that phase will be equal to 0. In this work, we select the least wetting phase to be the reference pressure, except in water-oil-gas systems where we choose the intermediate oil phase. Finally, to close the system, we apply constraints to ensure equal fugacities of each component in all phases where it is present \cite{Voskov2012EoSsimulation}.

\section{Numerical Discretization of Mass Balance Equations} \label{sec:NumDisc}
To numerically simulate flow and transport in a porous subsurface formation, we discretize the domain into $N$ control volumes. In this work, a fully implicit (backward-Euler) finite volume scheme is employed to translate the differential mass balance equations, eqs. \eqref{eq:DifferentialFormMassConservation}, onto our discrete system. Fully implicit discretizations are the most common formulation used in commercial simulators and allow for a concise and clear presentation. Note that HU schemes have shown significant efficiency benefits in sequential implicit schemes as well \cite{Moncorge2020Upwinding}. Future work will investigate the benefits of the presented scheme for sequential formulations. 

Discretizing eqs.\eqref{eq:DifferentialFormMassConservation}, the coupled system of equations with the discrete mass balances for every cell $i$, at time step $n+1$, reads:
\begin{align}
    V_i \phi_i  \frac{M_{c,i}^{n+1}-M_{c,i}^{n}}{\Delta t}  + \sum_{j \in con(i)} F_{c,ij}^{n+1} =  V_i Q_{c,i}^{n+1} 
    \quad \quad \forall \ c \in \{1,...,n_c\} \quad \forall \ i \in \{1,...,N\} 
    \label{eq:DiscreteSystem}
\end{align}
where $V_i$ is the volume of cell $i$ and $con(i)$ is the set of cells connected to cell $i$. We define a component mass-per-volume
\begin{align}
    M_{c} \equiv \sum_\ell^{n_p} x_{c,\ell} \rho_\ell S_\ell . 
\end{align}
where we have omitted the superscript for the time index and the subscript for the cell index. Henceforth, we will omit these indices from our notation for simplicity whenever they can be inferred. Similarly, a source component mass flux is defined
\begin{align}
    Q_{c} \equiv \sum_\ell^{n_p} x_{c,\ell} \rho_\ell \Tilde{q}^s_{\ell}.
\end{align}
Finally, the component mass fluxes capture the mass exchange between cells, i.e. 
\begin{align} \label{eq:FcDiscrete}
    F_{c} \equiv \sum_\ell^{n_p} x_{c,\ell} \rho_\ell u_\ell.
\end{align}
Here, the interface indices, e.g. $ij$ indicating the interface between cells $i$ and $j$, are also omitted for simplicity. Note that $u_\ell$ in eq. \eqref{eq:FcDiscrete}, or $u_{\ell,ij}^{n+1}$ to be precise, is a scalar quantity that represents the integral of the phase velocity over the interface. Furthermore, $F_c = F_{c,ij}>0$ indicates a positive mass flux from cell $i$ to cell $j$. As the flux terms are computed with interface fluid properties that depend on the flow state, they introduce additional complexity. Moreover, this can lead to issues in convergence when solving the system of equations \eqref{eq:DiscreteSystem} \cite{WANG2013TrustRegion, Li2015TrustRegion}. To alleviate issues encountered with standard formulations, this work proposes a physically motivated and numerically favorable computation of the component mass fluxes. To compare schemes and motivate their use, the following subsections will first describe different flux formulations and the nonlinear solver of choice. Then section \ref{sec:FluidPropertyTreatment} will delve into the possible treatments of the fluid properties: $x_{c,\ell}$, $\rho_\ell$ and $\lambda_\ell$.

\subsection{Standard Flux Formulation}
We label the conventional approach, and current industry-standard for fully implicit simulations, as the standard flux formulation and denote its variables with $^{st}$. As such, the component mass fluxes are computed as
\begin{align} \label{eq:standardMassFlux}
    F_{c}^{st} = \sum_\ell^{n_p} x_{c,\ell}^{st} \rho_\ell^{st} u_\ell^{st},
\end{align}
where the phase velocity is computed directly using Darcy's law. Discretizing eq. \eqref{eq:Darcy} results in:
\begin{align} 
    u_\ell^{st} = T \lambda_\ell^{st} \Delta \Phi_{\ell}, \label{eq:standardPhaseVelocity}
\end{align}
where $T$ is the two-point flux approximation (TPFA) interface transmissibility which captures grid and rock properties \cite{AzizSettari}. Although we do not consider multipoint flux approximations (MPFA), the methods proposed in this work are also applicable to them. Similarly, the novelties from this work also apply to multipoint upwinding schemes\cite{Kozdon2011MultiD,Lamine2015multidimensional,Hamon2020multidimensionalHU}. Furthermore, remembering that these are all quantities of interface $ij$, we specify that the phase potential difference is computed as  
\begin{align}
    \Delta \Phi_\ell = \Delta \Phi_{\ell,ij} = \Delta p_{\ell,ij} - g_{\ell,ij},
\end{align}
with $g_{\ell,ij}$ defined as the phase gravity potential difference over interface $ij$, i.e.
\begin{align}
    g_{\ell,ij} \equiv \rho_{\ell,ij} g \Delta z_{ij}.
\end{align}
Here, difference term subscripts indicate which variable is subtracted ($\Delta \bullet_{ij} = \bullet_i-\bullet_j$). 

Lastly, the density at the interface, $\rho_{\ell,ij}$, used to compute the gravity potential is often calculated as the arithmetic average of the two neighboring cells. When a phase is no longer present in one of the cells, the value of the cell in which the phase is still present is chosen. This leads to jumps in density value which can cause problems in converging to the solution. Hence, in this work, we compute the interface density with a saturation-weighted average. That is,
\begin{align}
    \rho_{\ell,ij} = \frac{S_{\ell,i} \rho_{\ell,i} + S_{\ell,j} \rho_{\ell,j}}{S_{\ell,i}+S_{\ell,j}}. \label{eq:SaturationAverageDensity}
\end{align}
To avoid division by zero or zero densities, an epsilon value is added to the saturation weights in implementation. As saturation-weighted averaging is beneficial and already in use in available simulators \cite{MRST}, we apply it to all of the presented formulations.

\subsection{Total Velocity Formulation}
Fractional flow formulations are well established for sequential simulations \cite{Moyner2018Sequential,Moncorge2018SFIcompositional} and discrete interface conditions \cite{VanDuijn1995IC, Alali2021IC}. Recently, the total velocity formulation -- the most common fractional flow formulation --  has also been used to enable significant benefits for fully implicit simulations in terms of accuracy \cite{Alali2021IC, BRENNER2020,brenner2020TV} and nonlinear convergence \cite{LEE2015HU,Hamon2016HU_Buoyancy, Watanabe2016SequentialHU,Jiang2019SFI, brenner2021sequential}. In this work we focus on the latter, where we exploit a fractional flow formulation to distinguish the parabolic total flow subproblem from the hyperbolic transport subproblem in the context of a fully implicit method. 

In the total velocity formulation, the phase velocities are computed using a total velocity variable. This allows us to identify mobilities that dominate the flow problem, designated $^F$, and those that primarily influence the transport problem, designated $^T$. Applying the described notation to the component mass flux, we obtain:
\begin{align} \label{eq:componentFluxTV}
    F_{c}^{tv} = \sum_\ell^{n_p} x_{c,\ell}^{T} \rho_\ell^{T} u_\ell^{tv},
\end{align}
where $u^{tv}$ is the phase velocity computed with the total velocity formulation. A total velocity is then defined as the sum of the phase velocities in the Darcy form (eq. \eqref{eq:Darcy}), 
\begin{align} \label{eq:totalvelocity}
    u_t \equiv \sum_\ell u_\ell = T \sum_\ell \lambda_\ell^F \Delta \Phi_\ell.
\end{align}
Using $u_t$, we can then rewrite \eqref{eq:standardPhaseVelocity} and compute the phase velocity as 
\begin{align}
    u_\ell^{tv} =  \frac{\lambda_\ell^{T}}{\lambda_t^{T}} u_t + T \sum_{m}^{n_p} \frac{\lambda_\ell^{T} \lambda_m^{T}}{\lambda_t^{T}} (g_m - g_\ell) + T \sum_{m}^{n_p} \frac{\lambda_\ell^{T} \lambda_m^{T}}{\lambda_t^{T}} (\Delta p_{cap,m} - \Delta p_{cap,\ell}), \label{eq:totalVelocityPhaseVelocity}
\end{align}
where total mobilities are equal to the sum of the phase mobilities, i.e. $\lambda_t = \sum_\ell \lambda_\ell$. In the case where $\lambda_t = 0$, we add an epsilon in implementation. Although we designate the transport mobilities $\lambda^T$ all the same here, note that they can be treated differently in each of the terms (see Section \ref{sec:TransportUpwinding} where different numerical treatments for the transport mobilities are explored). 

We highlight that Eq. \eqref{eq:totalVelocityPhaseVelocity} only redistributes the total velocity and does not change the sum of the phase velocities. I.e. $\sum_\ell u_\ell^{tv} = u_t$. Hence, for problems with incompressible immiscible fluids, the flow mobilities ($\lambda^F$) completely determine the flow subproblem and its pressure field solution. 

\subsection{Total Mass Flux Formulation}
In addition to the presented existing formulations, we propose a novel total mass flux formulation. For cases where the fluids may be compressible, this fractional flow formulation still allows for the strict separation of the flow subproblem from the transport fluid properties. In essence, this extends the total velocity formulation's properties on their original test problems (incompressible volume balance) to a more general mass balance framework. 

Summing the mass balance equations of each component, eq. \eqref{eq:DifferentialFormMassConservation}, we obtain the total mass balance equation which can be expressed as
\begin{align}
    \phi \ \frac{\partial \Bar{\rho}}{\partial t} + \nabla \cdot  f_t = q_t ,
    \label{eq:FlowFormMassConservation}
\end{align}
where $\Bar{\rho}$ is the mass-weighted average density, the total mass flux is defined as
\begin{align}
    f_t \equiv \sum_\ell f_\ell  = \sum_\ell \rho_\ell u_\ell , \label{eq:totalmassflux}
\end{align}
and $q_t$ is a total mass source term constructed similarly to the total mass flux in \eqref{eq:totalmassflux}.
Note that as the phase compositions sum up to unity, they are not explicitly present in the flow problem, eq. \eqref{eq:FlowFormMassConservation}, and only implicitly influence flow through an effect on phase densities.

Defining a fluid mass mobility $\Lambda_\ell = \rho_\ell \lambda_\ell$, we can now rewrite the phase mass fluxes $f_\ell$ in terms of the total mass flux, similar to eq. \eqref{eq:totalVelocityPhaseVelocity}: 
\begin{align}
    f_\ell^{tm} =  \frac{\Lambda_\ell^T}{\Lambda_t^T} f_t + T \sum_{m}^{n_p} \frac{\Lambda_\ell^T \Lambda_m^T}{\Lambda_t^T} (g_m - g_\ell) + T \sum_{m}^{n_p} \frac{\Lambda_\ell^T \Lambda_m^T}{\Lambda_t^T} (\Delta p_{cap,m} - \Delta p_{cap,\ell}),
    \label{eq:phasemassfluxTotForm} 
\end{align}
where the discrete total mass flux is then be computed as 
\begin{align} \label{eq:totalmassflux_discretized}
    f_{t}  = T \sum_\ell \Lambda_\ell^F \Delta \Phi_{\ell}.
\end{align}
Finally, using the phase mass fluxes, the discretized component flux is equal to:
\begin{align} \label{eq:componentflux_TM}
    F_{c}^{tm} = \sum_\ell^{n_p} x_{c,\ell}^{T} f_\ell^{tm}.
\end{align}
Note that the total mass flux is the defining physical quantity being transferred through the domain, whereas the total velocity only satisfies that definition for cases with incompressible fluids. Moreover, as will be shown later, the total mass formulation in combination with hybrid upwinding allows for a smooth solution space of the compressible flow subproblem. For this purpose, we remark that, by substituting constraint \eqref{eq:compFracConstraint}, the sum of component fluxes is equal to the sum of phase fluxes, i.e. total mass flux,
\begin{align} \label{eq:TMEqualsTV}
    F_t = \sum_c^{n_c} F_{c} = \sum_c^{n_c} \sum_\ell^{n_p} x_{c,\ell} f_\ell  =   \sum_\ell^{n_p} f_\ell = f_t .
\end{align}

\subsection{Newton-Raphson: the Nonlinear Solver of Choice} \label{sec:Newton}
The Newton-Raphson method, often referred to simply as Newton's method, is the nonlinear solver of choice in reservoir simulation. Newton solves the nonlinear system by iteratively taking linear steps along its derivative. If nonlinearities are strong and the time step is large, the linear updates may fail to converge to the solution or even diverge \cite{Jenny2009, WANG2013TrustRegion}. More generally, time step cuts and many iterations may be required in these cases. 

To alleviate this issue, the reservoir simulation community resorts to heuristic Newton improvement strategies. The most common are saturation update damping strategies, e.g. modified Appleyard saturation update chopping which is applied by Eclipse \cite{Eclipse}. Modified Appleyard chopping ensures that saturation values remain within the feasible range, that is $0 < S_\ell < 1$ (assuming an effective saturation or no residual saturations), and that updates are bounded by a set threshold, e.g. $\Delta S_\ell < 0.3$. Current discretization schemes frequently need tight thresholding to safeguard against divergence, however this can also slow down convergence when unnecessary. In our numerical test cases, section \ref{sec:TestCases}, we try not to restrict Newton excessively and pursue the most efficient nonlinear strategy possible. To enable larger time steps in these scenarios, Section \ref{sec:WA} and \ref{sec:HU} present a discretization scheme with less numerics-induced nonlinearity.

We mention here that other nonlinear strategies can be combined with our approach to further improve convergence.
Recent developments for the fully implicit simulation of flow in porous media include---but are not limited to---physics-based damping strategies \cite{Jenny2009,WANG2013TrustRegion,Li2015TrustRegion,Moyner2017TrustRegions}, ordering-based nonlinear solvers \cite{kwok2007potential,natvig2008fast,klemetsdal2019efficient}, nonlinear preconditioning \cite{FieldSplitPreconditionedInexactNewtonAlgorithms_LiuKeyes,skogestad2016two,dolean2016nonlinear,Klemetsdal2020Schwarz}, nonlinear multigrid \cite{toft2018full,la2018nonlinear,lee2020nonlinear}, continuation methods \cite{younis2010adaptively,jiang2018dissipation}, and Anderson acceleration \cite{lott2012accelerated,lipnikov2013anderson}.

\section{Treatment of Numerical Fluid Properties} \label{sec:FluidPropertyTreatment}
Fluid properties depend on pressure and composition in a highly nonlinear fashion. These complex relationships complicate the solution of the reservoir simulation equations. Additional strong non linearities (e.g. kinks) are introduced by the discretization.  Importantly, here we differentiate between the former category of physical nonlinearities and latter category of numerical nonlinearities. While the physical aspects of the problem are essential to capture, we would like to minimize the severity and number of numerics-induced features. To that end, this section describes several strategies to treat the fluid properties in the computation of discrete mass fluxes. The effect on the nonlinearities of the full problem will then be analyzed in section \ref{sec:Analysis}.

Although the total velocity and total mass flux formulations are lengthier and require the computation of a total parameter, they allow for a more suitable treatment of properties in the flow and transport subproblems. Several schemes have been proposed in recent literature. Table \ref{tab:upwindingChoice} summarizes the different upwinding techniques applied to the flow and transport subproblems, respectively, in different discretization schemes.

The schemes are named based on their upwinding strategies in the flow and transport subproblems. The schemes presented are: the standard PPU scheme, PPU; the phase upwinding schemes presented in \cite{Moncorge2020Upwinding}, PPU-PU and PU; the original hybrid upwinding strategy with PPU mobilities in flow, PPU-HU; and a novel hybrid upwinding strategy with weighted average mobilities, WA-HU. We emphasize that all presented schemes are monotone, consistent and in the limit, $V_i \xrightarrow[]{} 0 \ \forall \ i \in \{1,...,N\}$, converge to the same solution. Note that all the schemes with a fractional flow formulation can be applied to both a total velocity or total mass formulation. We highlight that the presented WA-HU scheme contains two key contributions of this work: (i) an improved and general Weighted Averaging and (ii) a HU framework for multiphase compositional fluids.

\setlength{\arrayrulewidth}{.5mm}
\begin{table}[htbp] 
\caption{Upwinding choices for the standard, flow and transport discrete fluid properties in different schemes. Standard variables are present in Eqs. \eqref{eq:standardMassFlux} and \eqref{eq:standardPhaseVelocity}. Flow variables determine the total parameters $u_t$ (Eq. \eqref{eq:totalvelocity}) or $f_t$ (Eq. \eqref{eq:totalmassflux_discretized}). Transport variables determine the component fluxes for the mass balance equations, see Eqs. \eqref{eq:componentFluxTV} and \eqref{eq:totalVelocityPhaseVelocity} for TV, or see Eqs. \eqref{eq:phasemassfluxTotForm} and \eqref{eq:componentflux_TM} for TM.} \label{tab:upwindingChoice}
\centering 
\begin{tabularx}{\textwidth}{l *{5}{Y}}

\toprule
{}                 & {\textbf{PPU}} & { \textbf{PPU-PPU}} & { \textbf{PPU-PU}} & { \textbf{(PPU-)HU}} & { \textbf{WA-HU}} \\
\midrule
\textbf{Formulation}   & Standard      & Frac. Flow           & Frac. Flow         &   Frac. Flow                  & Frac. Flow  \\
\midrule
\textbf{Standard $^{st}$}   &  PPU                &            &                   &                    &  \\
\textbf{Flow $^F$}     &                 & PPU               & PPU                   & PPU                   & WA \\
\textbf{Transport $^T$} &           & PPU               & PU                    & HU                    & HU \\
\bottomrule
\end{tabularx}
\end{table}



\subsection{Flow and Standard Upwinding Strategies}
\subsubsection{Phase Potential Upwinding (PPU)}
In PPU the phase parameters are upwinded with respect to the phase potential \cite{Sammon1988PPU}, that is
\begin{align}
    \lambda_{\ell,ij}^{PPU} =
    \begin{cases}
        \lambda_{\ell,i} \quad & \text{if   } \Delta \Phi_{\ell,ij} \geq 0 \\
        \lambda_{\ell,j} \quad & \text{otherwise}
    \end{cases}.
\end{align}
If we use PPU mobilities for both flow ($\bullet^F$) and transport ($\bullet^{T}$) variables (scheme PPU-PPU), the fractional flow formulation is equivalent to the standard PPU formulation. Although equivalent, such an implementation still provides value when one wants to split the pressure and transport problem. For example, for flux computation reasons at heterogeneous interfaces \cite{Alali2021IC} or in sequential settings (also see comments in section \ref{sec:PhaseUpwinding}).

\subsubsection{Weighted Averaging} \label{sec:WA}
In this work, we propose a novel phase-based weighted averaging (WA) for the flow discretization. This scheme simplifies and improves existing HU-WA methods \cite{Hamon2016HU_Buoyancy,Hamon2018HU_capillarity}, as well as extends its applicability to compositional, miscible and compressible fluids. 

Through the splitting of the flow and transport subproblems, the flow mobilities $\lambda_\ell^F$ (or mass mobilities $\Lambda_\ell^F$) need to primarily satisfy the monotonicity conditions of the flow problem. These are looser than the conditions of the coupled parabolic-hyperbolic problem and hence allow for a more physical and smooth choice of the properties governing the flow problem. 

Defining a weighting coefficient $\beta_{\ell,ij}$ for every phase at the interface between cell $i$ and $j$, the flow mobilities, $\lambda_{\ell,ij}^{WA}$, can be computed as
\begin{align} \label{eq:weightedAverageMobility}
    \lambda_{\ell,ij}^{WA} = \beta_{\ell,ij} \lambda_{\ell,i} + (1-\beta_{\ell,ij}) \lambda_{\ell,j}.
\end{align}
Note that we use the total velocity formulation with mobilities throughout this description, however, the process is identical when applied to mass mobilities in a total mass formulation. 

$\beta_{\ell,ij}$ specifies how smoothly to transition from one upwind direction to the other. Phase velocities can change direction often if counter-current forces, such as capillary and gravity potentials, are large. Conversely, changes in the sign of a phase velocity are rare in viscous dominant regimes which are co-current by nature. Additionally, when a phase potential or a phase velocity is large, that phase velocity is less likely to change sign.  
To this end, we define 
\begin{align} \label{eq:beta}
    \beta_{\ell,ij} = 0.5 + \frac{1}{\pi} \arctan\Bigg(\gamma_\ell \ \frac{\Delta \Phi_{\ell,ij}}{|{g}_{ij,ref}| + |{c}_{\ell,ref}|}\Bigg),
\end{align}
where the phase potential is normalized by a reference gravity potential and reference capillary potential. In this work we set these to be: the face gravity potential computed with the maximum surface phase density ($\rho_{\ell,s}$),
\begin{align}
    {g}_{ij,ref} = \max_\ell(\rho_{\ell,s}) \boldsymbol{g} \Delta z_{ij},
\end{align}
and the capillary pressure difference between a saturation of $0.2$ and $0.8$, ${c}_{\ell,ref} = p_{\ell,cap}(0.8)-p_{\ell,cap}(0.2)$. Note that we choose values that do not change during the simulation. Here, surface phase densities are used to be a representative density value. Other approximate choices are also valid. The goal of ${c}_{\ell,ref}$ is to capture a representative capillary pressure difference that two neighboring cells might encounter in the simulation. We also tested using the simulation values of the gravity and capillary potential instead of reference values. However, we experienced little change in nonlinear efficiency in addition to higher computational cost. Another considered strategy used lagged simulation values as the reference values.

The arctan function, proposed in \cite{Hamon2016HU_Buoyancy}, is chosen as it allows for scheme monotonicity while providing a smooth transition between upwind directions. Lastly $\gamma_\ell$ is a scaling coefficient which determines how close the mobilities will be to PPU mobilities ($\gamma \xrightarrow{} \infty$) or arithmetically averaged mobilities ($\gamma = 0$). Hence, we design $\gamma_\ell$ such that it trades off the negative aspects of PPU and AA. As discussed, PPU schemes introduce a kink when switching the upwind direction. On the other hand, AA schemes lead to downwind dependence and hence significant curvature in the total velocity, or mass flux, profile with respect to downwind saturation (also see section \ref{sec:Analysis} for visual examples of kinks and curvature). We quantify these effects as
\begin{align}
    \text{bend}^{AA}_\ell = \max_{S_\ell}\bigg|\frac{\partial^2 \lambda_\ell}{\partial S_\ell^2}\bigg| = \frac{\rho_\ell}{\mu_\ell} \max_{S_\ell}\big|k_{r,\ell}''\big|
\end{align}
with $k_{r,\ell}''$ the second derivative of the relative permeability function, and 
\begin{align}
    \text{kink}^{PPU}_\ell = \max_{S_{\ell,1},S_{\ell,2}} \Big|\lambda_\ell(S_{\ell,1})-\lambda_\ell(S_{\ell,2}) \Big| = \frac{\rho_\ell}{\mu_\ell} k_{r0,\ell},
\end{align}
where $k_{r0,\ell}$ is the phase relative permeability end point. We then define a scaling coefficient, $\gamma_\ell$, to capture the ratio of the nonlinearities;
\begin{align}
    \gamma_\ell = \alpha \frac{\text{bend}^{AA}_\ell}{\text{kink}^{PPU}_\ell} = \frac{\alpha}{k_{r0,\ell}} \max_{S_\ell}\big|k_{r,\ell}''\big|
\end{align}
where $\alpha$ is an optional user tuning parameter (set to $1$ in this work).

The presented weighted averaging is a generalized and improved version of the method in \cite{Hamon2016HU_Buoyancy}, yet still retains the same structure. Specifically, the scheme is extended to compositional simulations and improved in terms of more physical and effective smoothing. This is achieved while retaining the same functions in eq. \eqref{eq:weightedAverageMobility} and eq. \eqref{eq:beta}. As such, the propositions on saturation estimate, pressure estimate, solution existence, bounds on pressure matrix coefficients, invertability of the pressure and Jacobian matrix of the original WA scheme still hold. Finally, \ref{app:Proof} presents the proof of monotonicity of the total velocity with respect to pressure. The proof specifies that monotonicity is guaranteed when the weight in the PPU upwind direction is greater than, or equal to, the weight in downwind direction. That is
\begin{align}
    \begin{array}{l}
         \beta_{\ell,ij} \geq 0.5   \quad  \text{if   } \Delta \Phi_{\ell,ij} \geq 0,  \\
         \beta_{\ell,ij} < 0.5   \quad  \text{otherwise}
    \end{array} 
\end{align}
or equivalently, following eq.\eqref{eq:beta}, $\gamma_\ell \geq 0$. Specifying $\alpha \geq 0$, this condition is satisfied by construction. Hence, the scheme is unconditionally monotone. 

\subsection{Transport Upwinding Strategies} \label{sec:TransportUpwinding} 
\subsubsection{Hybrid Upwinding (HU)} \label{sec:HU}
Hybrid upwinding aims to improve non-linear solver convergence by upwinding the transport variables more appropriately for the respective physics. Initially presented for two-phase flow \cite{LEE2015HU}, recent works have aimed to effectively extend the hybrid upwinding methodology to 3-phase flow \cite{Hamon2018HU_capillarity, Lee2016C1continuous3phase,  Lee2018HU3phaseCapillarity} and compositional simulations \cite{Moncorge2020Upwinding}. We build on these works and present a compositional hybrid upwinding framework that closely resembles that of Moncorg\'e et al. \cite{Moncorge2020Upwinding}.

Hybrid upwinding defines a specific upwinding for each of the forces. To this end, we specify superscripts indicating viscous($^V$), gravity($^G$) and capillary($^C$) terms. Applying these to eq. \ref{eq:totalVelocityPhaseVelocity}, or similarly to eq. \ref{eq:phasemassfluxTotForm}, we obtain
\begin{align}
    u_\ell^{tv} =  \frac{\lambda_\ell^{V}}{\lambda_t^{V}} u_t + T \sum_{m}^{n_p} \frac{\lambda_\ell^{G} \lambda_m^{G}}{\lambda_t^{G}} ({g}_m - {g}_\ell) + T \sum_{m}^{n_p} \frac{\lambda_\ell^{C} \lambda_m^{C}}{\lambda_t^{C}} (\Delta p_{cap,m} - \Delta p_{cap,\ell}). \label{eq:HUtotalvelocity}
\end{align}
The viscous mobilities are then upwinded with respect to the total velocity
\begin{align}
    \lambda_{\ell,ij}^{V} =
    \begin{cases}
        \lambda_{\ell,i} \quad & \text{if   } u_{t,ij} \geq 0 \\
        \lambda_{\ell,j} \quad & \text{otherwise}
    \end{cases}.
\end{align}
For the gravity and capillary forces, we apply an upwinding based on potential ordering for each of the forces separately. This treatment is a variation of the method proposed by Brenier and Jaffr\'e in \cite{Brenier1991Upwind}, where the original analysis encompasses this option. For the upwinding, we define a gravity upwinding parameter:
\begin{align}
    \omega_\ell^G = \sum_{m}^{n_p} \big[(\lambda_{m,i} \cdot({g}_m < {g}_\ell) + \lambda_{m,j} \cdot({g}_m > {g}_\ell)\big] ({g}_m - {g}_\ell). \label{eq:gravityUpwindingParameter}
\end{align}
The sign of this upwinding parameter indicates the upwind direction for the mobilities of the gravity term, i.e.
\begin{align}
    \lambda_{\ell,ij}^{G} =
    \begin{cases}
        \lambda_{\ell,i} \quad & \text{if   } \omega_\ell^G \geq 0 \\
        \lambda_{\ell,j} \quad & \text{otherwise}
    \end{cases} .
\end{align}
Similarly, we define a capillary upwinding parameter, replacing the gravity potentials in eq.\eqref{eq:gravityUpwindingParameter} with $\Delta p_{cap}$, and upwind accordingly 
\begin{align}
    \lambda_{\ell,ij}^{C} =
    \begin{cases}
        \lambda_{\ell,i} \quad & \text{if   } \omega_\ell^C \geq 0   \\
        \lambda_{\ell,j} \quad & \text{otherwise}
    \end{cases}.
\end{align}
Physically speaking, the mobilities of the viscous term are chosen upstream of the total flux. The gravity properties of the heavier phases are chosen from the higher cell and vice-versa for the lighter phases (i.e. from the lower cell). The mobilities of the capillary term are chosen towards the higher saturation for the least wetting phases. 

Considering the two phase water-gas case: gravity gas mobilities would be upwinded to the lower cell, and capillary gas mobilities would be upwinded to the cell with lower gas saturation. The water mobilities are then upwinded in the opposite directions. If an additional intermediate phase is introduced, the presented scheme selects its upwind direction based on the balance of mobility potentials (see eq. \ref{eq:gravityUpwindingParameter}).

Finally, we specify the upwinding for the fluid properties outside of the total parameter: $x_{c,\ell} \rho_\ell$ for total velocity (see eq. \eqref{eq:componentFluxTV}) or only $x_{c,\ell}$ for total mass flux (see eq. \eqref{eq:TMEqualsTV}). Here we have two options. Either we upwind with respect to the total parameter or upwind for each of its components separately. Although our test cases did not indicate a significant difference, in this work we implement the latter as we favor influencing three existing small kinks (at V, G or C flips) over creating one new large one at phase (mass-)flux flips. Defining the three velocity terms in eq. \eqref{eq:HUtotalvelocity} as $V$, $G$ and $C$, eq. \eqref{eq:componentFluxTV} can be rewritten as
\begin{align} \label{eq:fluxWithTermNotation}
    F_{c}^{tv} = \sum_\ell^{n_p} \bigg( x_{c,\ell}^{V} \rho_\ell^{V} V + x_{c,\ell}^{G} \rho_\ell^{G} G + x_{c,\ell}^{C} \rho_\ell^{C} C \bigg).
\end{align}
Here, each term is upwinded with respect to the corresponding velocity component. E.g.
\begin{align} \label{eq:compositionUpwindingEachTerm}
    x_{c,\ell}^{V} \rho_\ell^{V}  =
    \begin{cases}
        (x_{c,\ell} \rho_{\ell})_i  \quad & \text{if   } V \geq 0 \\
        (x_{c,\ell} \rho_{\ell})_j  \quad & \text{otherwise}
    \end{cases} .
\end{align}
Note that although all the equations in this section are presented with the total velocity formulation, the methodology can be extended in a straightforward manner to the total mass formulation. 

\subsubsection{Phase Upwinding (PU)} \label{sec:PhaseUpwinding}
Phase Upwinding was proposed for sequential simulations, however, for completeness we include it in this overview. Applying Phase-Potential Upwinding (PPU) in sequential settings -- that is a PPU-PPU scheme -- can lead to downwinding for the second subproblem. Phase upwinding \cite{Moncorge2020Upwinding} resolves this issue by upwinding with a consistent condition defined by Brenier and Jaffr\'e \cite{Brenier1991Upwind} where potential ordering is applied to the viscous, gravity and capillary forces together. Note that in the fully implicit method, consistency between the subproblems is guaranteed and as such PPU-PU, PPU-PPU and standard PPU are identical. 

A PU-PU scheme has also been proposed for sequential simulations. In that case, PU upwinding from the previous transport problem is also used for the next flow problem. This can damp oscillatory flow updates between sequential iterations and improve convergence. However, due to the lag in upwinding, PU-PU generally leads to slower nonlinear convergence in fully implicit simulations \cite{Moncorge2020Upwinding}.

\section{Analysis of Discrete Residual Space} \label{sec:Analysis}
To illustrate the effect of the presented discretization schemes, we analyse the total velocity profile and the continuous Newton paths in the discrete residual space. For this purpose, a one-cell water-gas test case is designed. The design of the test case allows us to fully visualize and understand the problem. As the problem is for illustration purposes we omit any dimensional units. Figure \ref{fig:setup1cell} displays the problem set-up. We solve for the variables of cell $L$ which is surrounded by no-flow boundaries except for fixed boundary conditions on its right ($R$): $p_{w,R}=210$ and $S_{w,r}=0.2$. The domain is at an angle such that cell $R$ is located higher than cell $L$. Grid and rock properties are chosen to be: permeability $k=1$, porosity $0.3$, interface area $A=1$, distance $\Delta x=1$ and gravity multiplied by depth difference $g\Delta z = -1$. Furthermore we define fluid properties for the two phases, water and gas, respectively: mobilities $\lambda_w = S_w^2$ and $\lambda_g = S_g^3$, fixed densities $\rho_{w,L}=6.18$, $\rho_{w,R}=6$, $\rho_{g,L}=2.06$ and $\rho_{g,R}=2$. Note that we simplistically vary densities to show the effect of compressibility and the corresponding kinks at changes in upwinding directions of the various flux terms (see eq. \eqref{eq:compositionUpwindingEachTerm}). We use a table input with spline interpolation for the capillary pressure relationships. The table input can be found in \ref{app:Capillary}. In this proxy case we multiply the input data by $5 \times 10^{-5}$ such that the acting forces, gravity and capillarity, are of comparable magnitude. \ref{app:Capillary} also reasons for a couple good practices with respect to the implementation of capillary relationships.

\begin{figure}[htbp]
    \centering
    \includegraphics[width=.3\textwidth]{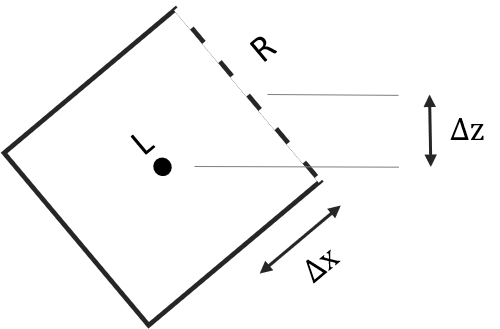}
    \caption{one-cell problem set-up.}
    \label{fig:setup1cell}
\end{figure}

Figure \ref{fig:TotalVelocity} compares the total velocity profiles of the PPU and WA schemes. The PPU profile clearly contains kinks at flips in phase velocity (black lines). These kinks are no longer present in the WA profiles. Importantly, we also show how the new proposed WA scheme still closely resembles the PPU profile. In contrast the originally proposed WA scheme \cite{Hamon2016HU_Buoyancy, Hamon2018HU_capillarity} deviates significantly from PPU and requires additional smoothing to satisfy its monotonicity requirements. This notable difference illustrates that the new scheme resolves issues brought forward in \cite{JIANG2017C1PPU}, where the WA flux seemed significantly different from the PPU flux. The reason for this difference is the increased flexibility of the weighting coefficients in the new scheme. Specifically, the new scheme can arithmetically average the phase mobility when near a phase's velocity flip (black lines), and take a single phase potential upwind direction otherwise. This allows for a smooth transition near the black lines, as well as almost no downwind saturation curvature (see profile along the saturation axis). 

Moreover, the black lines delineate the counter-current flow regime. As solutions in this regime are by nature close to the phase velocity flips, Newton iterations will cross them more often in counter-current flow problems. Consequently we see that arithmetic averaging is generally more efficient at solving counter-current flow problems. Vice-versa, deep in the co-current flow regimes, PPU outperforms AA due to the linear profile for the downwind cell. Weighted averaging combines both strategies and selects efficient upwinding weights for each interface in the discrete grid. This leads to great performance in both flow regimes as shown in the numerical examples of Section \ref{sec:TestCases}. 

\begin{figure} [htbp]
\centering
\begin{subfigure}[t]{0.49\textwidth}
\centering
\includegraphics[width=\linewidth]{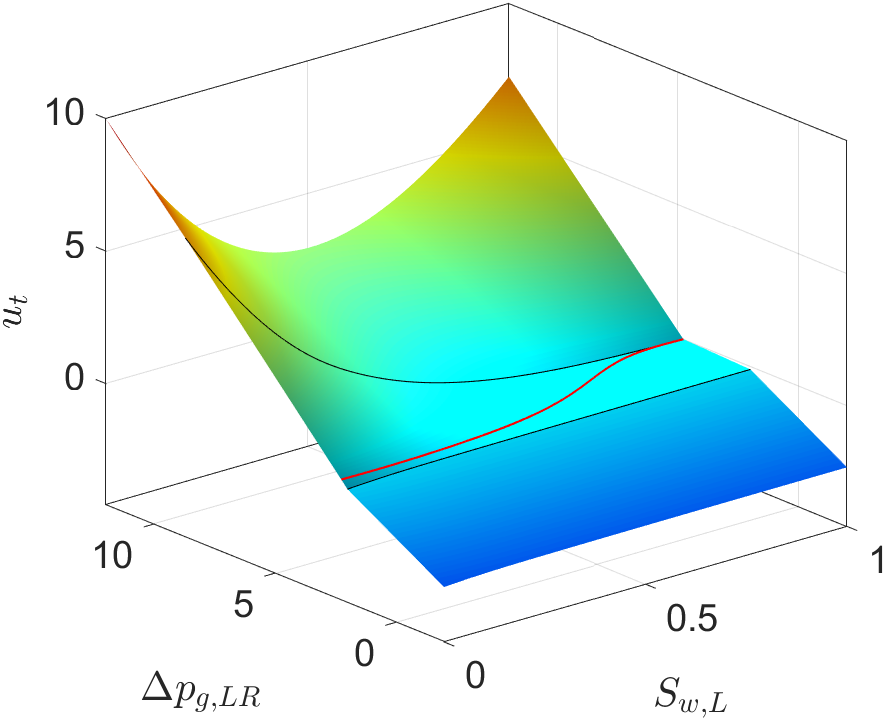}\hfill 
\caption{\label{fig:TotalVelocityPPU} PPU}
\end{subfigure}
\hfill
\begin{subfigure}[t]{0.49\textwidth}
\centering
\includegraphics[width=\linewidth]{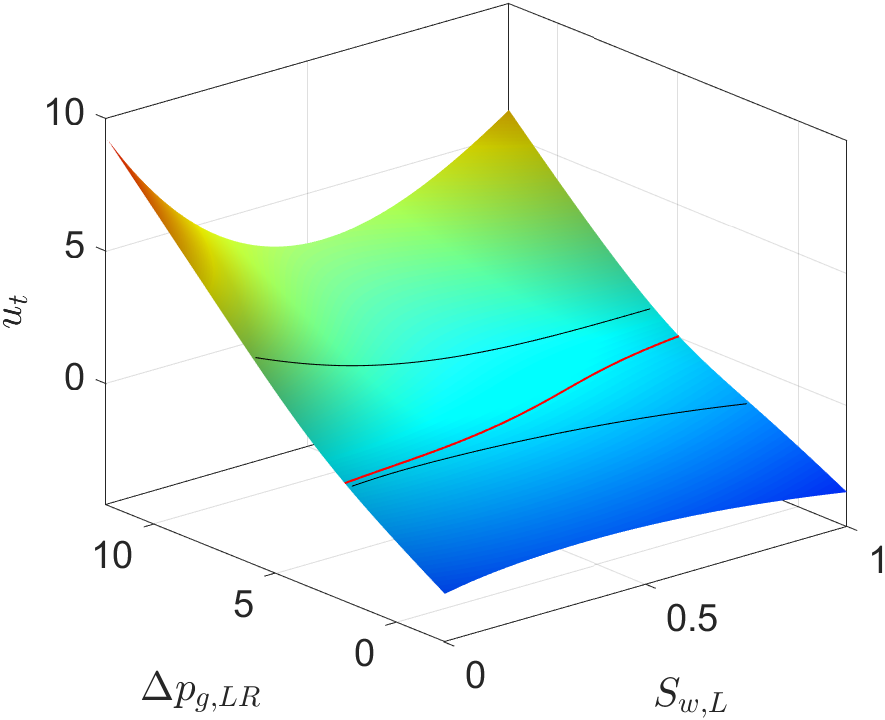}\hfill 
\caption{\label{fig:TotalVelocityWAold} WA (Hamon et al. \cite{Hamon2018HU_capillarity})}
\end{subfigure}
\hfill
\begin{subfigure}[t]{0.49\textwidth}
\centering
\includegraphics[width=\linewidth]{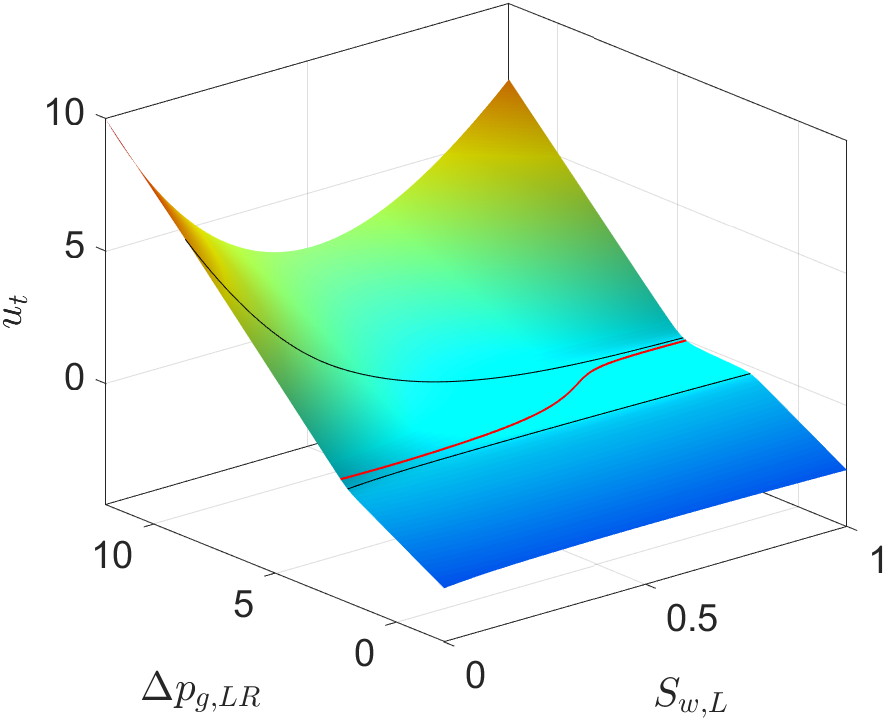}\hfill 
\caption{\label{fig:TotalVelocityWAnew} WA}
\end{subfigure}
\caption{\label{fig:TotalVelocity} Total velocity profiles of the one-cell proxy test case for the PPU, WA \cite{Hamon2018HU_capillarity} and new WA scheme. The left cell pressure and saturation are varied, while the right cell boundary conditions remain fixed. Phase potential flips and total velocity flips (i.e. changes in direction) are indicated with black and red lines, respectively. Note that in the original WA scheme, capillary pressures are excluded from the phase potential flips because they are treated separately. }
\end{figure}

The effect of these kinks on Newton convergence becomes more evident when investigating the solution space and Newton paths. We illustrate the solution space by plotting the $L_2$-norm of the residual. Furthermore, a Newton path is defined as the continuous path from an initial guess to the problem solution. It is constructed by taking small steps along the Newton step direction at each subsequent location. To compute and plot the above, we additionally specify the previous left cell saturation $S_{w,l}^n = 0.4$ at time $n$ (assuming the solution is at $n+1$) and time step $\Delta t = 0.1$. 

The resulting Newton paths and $L_2$ residual contours are plotted in Figure \ref{fig:NewtonPath}. We emphasize discontinuities in the derivative of the residual (strong nonlinearities known to be problematic for Newton); upwinding direction changes are indicated for the phase velocities (blue), total velocity/mass flux (red) and capillary term (green) when relevant. I.e. the phase direction changes are only indicated in these plots if they induce a non-linearity in the solution space of the corresponding scheme. 

From experience \cite{Hamon16HU_Analysis, Hamon2016HU_Buoyancy} and intuition, we know straighter and $C^1$ continuous Newton paths lead to quicker convergence. Figure \ref{fig:NewtonPath} confirms that the proposed WA-HU scheme results in favourable properties. Importantly, the most severe kinks are removed from the solution space and the paths are more direct. It is interesting to point out that the non-linearity corresponding to a capillary term flip (in green) is rather small. This is due to the fact that the flip happens at $S_R=S_L$ and hence the $C$ term (see eq. \eqref{eq:fluxWithTermNotation}) is symmetric (and smooth) for TV schemes and close to symmetric for TM schemes. Nevertheless, we note that these kinks can be more present when taking into account compositional upwinding at a sharp front. 

Finally, we highlight that there is an important difference between TV and TM schemes. In TV schemes the red line indicates a total velocity flip ($u_t=0$), whereas in TM schemes the red line represents the total mass flux flip ($f_t=0$). As this is a gravity segregation problem, we know the solution consists of a volume-balance and hence satisfies $u_t=0$. Moreover, in practice, solutions of this kind occur more often than total mass flux balances. Consequently, as displayed in figure \ref{fig:NewtonPathWAHU_TM}, TM solutions are more likely to be close to, but not on, a kink. As will be shown in the test case results, this proximity can lead to complications in convergence for TM schemes.

\begin{figure} [htbp]
\centering
\begin{subfigure}[t]{0.48\textwidth}
\centering
\includegraphics[width=\linewidth]{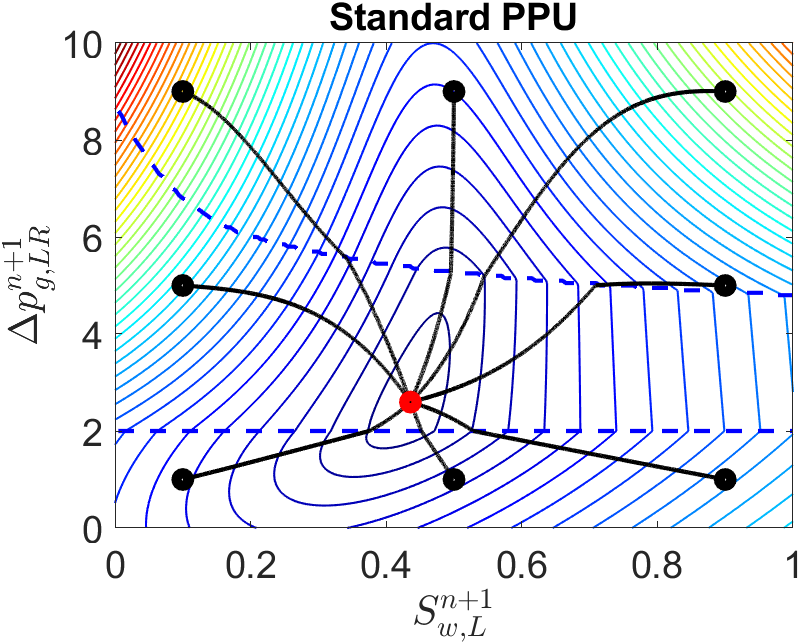}\hfill 
\caption{\label{fig:NewtonPathPPU}}
\end{subfigure}
\hfill
\begin{subfigure}[t]{0.48\textwidth}
\centering
\includegraphics[width=\linewidth]{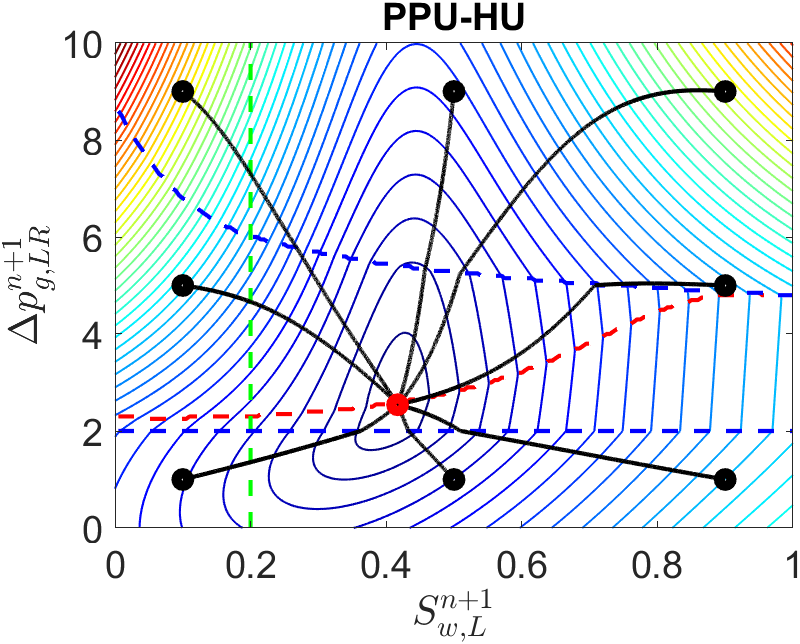}\hfill 
\caption{\label{fig:NewtonPathPPUHU}}
\end{subfigure}
\hfill
\vspace{.4cm}
\begin{subfigure}[t]{0.48\textwidth}
\centering
\includegraphics[width=\linewidth]{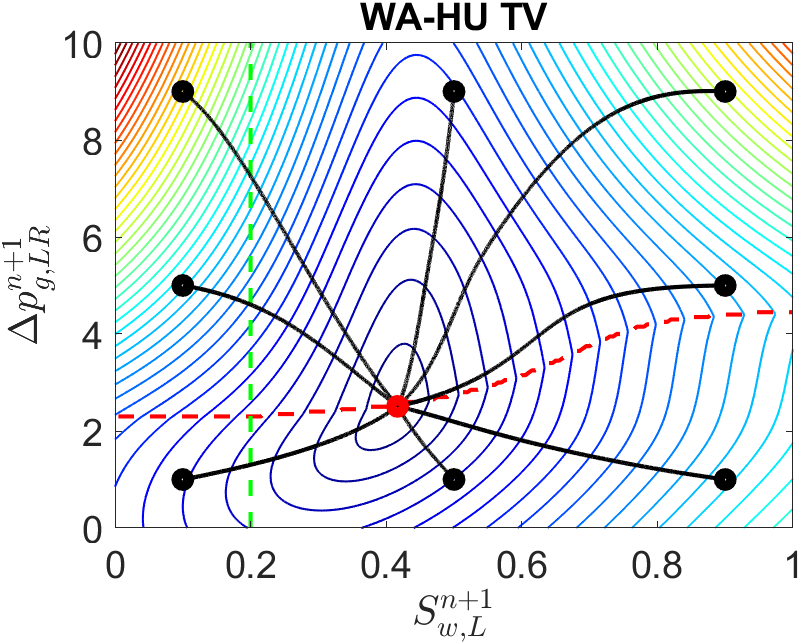}\hfill 
\caption{\label{fig:NewtonPathWAHU}}
\end{subfigure}
\hfill
\begin{subfigure}[t]{0.48\textwidth}
\centering
\includegraphics[width=\linewidth]{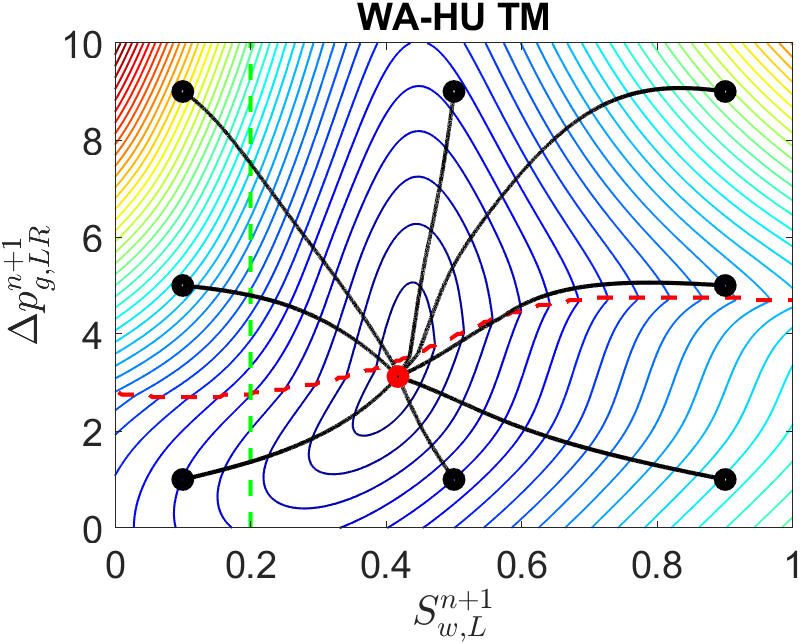}\hfill 
\caption{\label{fig:NewtonPathWAHU_TM}}
\end{subfigure}
\hfill
\caption{\label{fig:NewtonPath} Residual contours (in $L_2$-norm) and Newton paths for (a) PPU, (b) PPU-HU and (c) WA-HU with a total velocity formulation, and (d) WA-HU with a total mass formulation. The Newton paths connect the initial guesses (black dots) to the solution (red dot). The dashed red line indicates $u_t = 0$ for TV and $f_t = 0$ for TM. The dashed blue lines show $\Delta \Phi_\ell = 0$ and the dashed green line $\Delta p_{cap,w} = \Delta p_{cap,g}$.  }
\end{figure}

\section{Numerical Test Cases} \label{sec:TestCases}
To assess the proposed schemes' performance, several test cases are evaluated. The cases are designed with increasing spatial and physical complexity to demonstrate the WA-HU scheme's performance across the board. To this end, we look for the scheme which consistently produces low Newton iteration counts. Test cases \ref{sec:1Dtest} and \ref{sec:2DTiltedBox} are run in an in-house simulator. Test cases \ref{sec:barriers}, \ref{sec:Plume}, \ref{sec:fractures} and \ref{sec:SPE10} are run using an implementation in the development version of MRST \cite{MRST,Krogstad2015MRST}. We use mass variables for the compositional simulations, and natural variables otherwise. Note that the proposed schemes can be applied to both choices of primary variables.

As the convergence settings are similar throughout, we briefly describe them here. If a time step does not converge after $20$ iterations (or $15$ iterations for \ref{sec:1Dtest} and \ref{sec:2DTiltedBox}), two time steps half as large are taken. Convergence is reached at a normalized residual norm of $||r||_2 < 10^{-6}$, where we normalize with respect to cell fluid mass. For test cases \ref{sec:1Dtest} and \ref{sec:2DTiltedBox}, saturation change and relative pressure change are also checked for all cells: $|S^{k}_i-S^{k-1}_i| <0.01$ and $\frac{|p^{k}_i-p^{k-1}_i|}{p_i^{k}} < 10^{-3}$, where $k$ is the iteration number. To maximally challenge the non linear solver we refrain from using heuristic Newton damping where possible. Saturation updates are only chopped when they lead to saturation values that exceed the physical boundaries $0<S_{\ell}<1$. This is also commonly referred to as "vanilla" Newton. As test cases \ref{sec:fractures} and \ref{sec:SPE10} are significantly harder than the other test cases, we apply Modified Appleyard saturation update damping with a maximum update of 0.2 (also see Section \ref{sec:Newton}).

\subsection{Incompressible, Immiscible 1D Gravity Segregation without Capillary Forces} \label{sec:1Dtest}
The first test case is a one-dimensional (1D) gravity segregation problem. The problem contains two fluids: a heavy phase (designated as water $w$) and a light phase (designated as gas $g$). At initialization the heavy fluid occupies the top-half of the domain and the light fluid the bottom half. Corey relative permeability functions are chosen: $k_{r,w} = S_w ^{2.5}$ and $k_{r,g} = 0.6 S_g ^{3}$. The fluid densities are $\rho_w = 1000 [kg/m^3]$ and $\rho_g = 100 [kg/m^3]$ and their viscosities both set to $\mu_w = \mu_g = 0.001 [kg/(m.s)]$. The rock permeability and porosity are chosen to be $k =100$ [mDarcy] and $\phi =0.25$, respectively. A domain of dimensions $10$m$\times 10$m$ \times 200$m is discretized into $100$ cells in the z-direction. The simulation is run with four different fixed time step sizes of $100$, $150$, $200$, and $300$ days until approximate steady state which is attained at $ 5000 $ days. As such the simulations have maximum Courant CFL numbers \cite{FRANC2016CFL} of approximately $1$, $1.5$, $2$, and $2.7$, respectively. The first three time steps in all simulations are chosen to be smaller, that is $5$, $25$ and $50$ days, to avoid convergence issues related to initialization.

Table \ref{tab:results1D} summarizes the results of the 1D gravity segregation test case. The iteration counts verify that all proposed schemes perform similarly to the state-of-the-art on a small "simple" problem. Moreover, they also indicate that the HU schemes are significantly more efficient for larger time step sizes by both converging more rapidly and avoiding divergent time steps.

\begin{table}[htbp] 
\caption{Total number of Newton iterations performed for the one-dimensional gravity segregation test case. The wasted Newton iterations due to time step cuts are included in the totals and also displayed between parentheses.} \label{tab:results1D}
\centering 
\begin{tabularx}{.9\textwidth}{l L R L R L R L R}

\toprule
{Time step size}                 & \multicolumn{2}{ c }{\textbf{PPU}}  & \multicolumn{2}{ c }{ \textbf{(PPU-)HU}} & \multicolumn{2}{ c }{ \textbf{WA-HU TV}} & \multicolumn{2}{ c }{ \textbf{WA-HU TM}} \\
\midrule

\textbf{$\Delta t_{max} = 100$} days & 243&(0)               &  241&(0)                 &  239&(0)            & 231&(0)                   \\
\textbf{$\Delta t_{max} = 150$} days &  238&(30)               & 189&(0)                 & 192&(0)                & 181&(0)                 \\
\textbf{$\Delta t_{max} = 200$} days &  262&(75)               & 187&(15)                 & 190&(15)                & 185&(15)                 \\
\textbf{$\Delta t_{max} = 300$} days &  407&(210)               & 357&(180)                     & 288&(105)                 &254&(90)                      \\
\bottomrule
\end{tabularx}
\end{table}

\subsection{Incompressible, Immiscible Tilted Box: Lock Exchange to Gravity Segregation} \label{sec:2DTiltedBox}
The second series of test cases aims to test flux dynamics across faces at different angles. To do so, a 2-D square domain is initialized with $80$\% of the cells saturated with water (heavier fluid) and the other $20$\% saturated with gas (lighter fluid). Five different cases are considered where the square box and initial conditions are rotated from gravity segregation ($0^{\circ}$) to lock-exchange ($90^{\circ}$). Note that for a tilting angle of $0^{\circ} < \theta < 90^{\circ}$, all the interfaces have a non-zero $\Delta z$ which makes the test more challenging. Figure \ref{fig:TiltedBoxSnaps} depicts the initial condition and physical process with the domain at $70^{\circ}$. 

\begin{figure}[htbp]
    \centering
    \includegraphics[width=.9\textwidth]{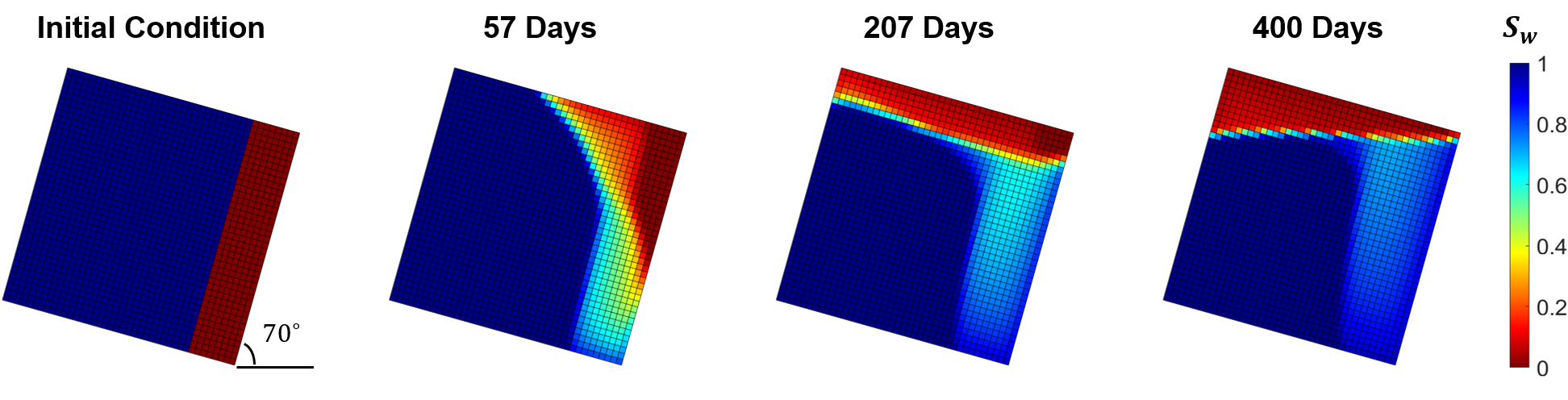}
    \caption{Snapshots of the water saturation throughout the simulation with a tilt angle $70^\circ$.}
    \label{fig:TiltedBoxSnaps}
\end{figure}

The rock and fluid properties are identical to the 1D test case. The domain is sized to be $20$m$\times 1$m $\times 20$m and is discretized into $40 \times 1 \times 40$ cells. We run the simulation for 400 days with time steps of 10 days. Again the first sequence of time steps is chosen to be smaller: 0.1, 0.5, 2 and 5 days. 

The test cases are run both with and without capillary forces. The capillary pressure relationship curves are represented using spline interpolation on the data presented in \ref{app:Capillary}. We choose this relationship as it avoids asymptotes which have little influence on the results but can hinder efficient convergence. Tables \ref{tab:resultsBoxNo} and \ref{tab:resultsBoxwithCap} present the cumulative Newton iteration counts for the simulations without and with capillary forces, respectively. Without capillary forces, the WA-HU scheme in a total mass formulation is the fastest to converge in most set-ups. However, the scheme can also run into significant convergence issues, as is the case for the $70^{\circ}$ tilted box. The time step cuts all take place at two distinct time steps where the solution for four cells lies close to the mass flux balance ($F_t=0$, eq. \eqref{eq:TMEqualsTV}). As explained in section \ref{sec:Analysis}, proximity to a kink -- in this case the dashed red line in figure \ref{fig:NewtonPathWAHU_TM} -- can complicate Newton's ability to converge efficiently. This test case gives a first indication that the location of the total mass flux kink in the mass formulation creates more convergence difficulties than the total velocity kink in the velocity formulation (see the red dashed lines in figure \ref{fig:NewtonPath}).

In a total velocity formulation, the WA-HU scheme does not exhibit the same risk of diverging, while delivering comparable benefits. In fact, in the case with capillarity, these benefits are only present for the WA-HU TV scheme. More generally, the differences in efficiency are less important with capillary forces because the saturation front moving through the domain is more diffused. Overall, looking both at robustness and efficiency, the tilted box series of test cases motivate using the WA-HU TV scheme.

\begin{table}[htbp] 
\caption{Total number of Newton iterations performed for the tilted box test cases, only including gravity forces. The wasted Newton iterations due to time step cuts are included in the totals and also displayed between parentheses. The angle at which the box is tilted varies from $0^{\circ}$ (gravity segregation) to $90^{\circ}$ (lock-exchange).} \label{tab:resultsBoxNo}
\centering 
\begin{tabularx}{.9\textwidth}{l L R L R L R L R}

\toprule
{Box Tilt Angle}                 & \multicolumn{2}{ c }{\textbf{PPU}}  & \multicolumn{2}{ c }{ \textbf{(PPU-)HU}} & \multicolumn{2}{ c }{ \textbf{WA-HU TV}} & \multicolumn{2}{ c }{ \textbf{WA-HU TM}} \\
\midrule

\textbf{$0^{\circ}$} & 158&(0)               &  159&(0)                 &  157&(0)            & 154&(0)                   \\
\textbf{$20^{\circ}$} &  255&(0)               & 218&(0)                 & 212&(0)                & 207&(0)                 \\
\textbf{$45^{\circ}$} &  281&(0)               & 241&(0)                 & 230&(0)                & 221&(0)                 \\
\textbf{$70^{\circ}$} &  272&(0)               & 254&(15)               & 227&(0)            &704&(480)                      \\
\textbf{$90^{\circ}$} &  273&(15)               & 247&(15)               & 227&(0)            &227&(0)                      \\
\bottomrule
\end{tabularx}
\end{table}

\begin{table}[htbp] 
\caption{Total number of Newton iterations performed for the tilted box test cases with gravity and capillary forces. The wasted Newton iterations due to time step cuts are included in the totals and also displayed between parentheses. The angle at which the box is tilted varies from $0^{\circ}$ (gravity segregation) to $90^{\circ}$ (lock-exchange).} \label{tab:resultsBoxwithCap}
\centering 
\begin{tabularx}{.9\textwidth}{l L R L R L R L R}

\toprule
{Box Tilt Angle}                 & \multicolumn{2}{ c }{\textbf{PPU}}  & \multicolumn{2}{ c }{ \textbf{(PPU-)HU}} & \multicolumn{2}{ c }{ \textbf{WA-HU TV}} & \multicolumn{2}{ c }{ \textbf{WA-HU TM}} \\
\midrule

\textbf{$0^{\circ}$} & 165&(0)               &  163&(0)                 &  160&(0)            & 161&(0)                   \\
\textbf{$20^{\circ}$} &  209&(0)               & 210&(0)                 & 207&(0)                & 209&(0)                 \\
\textbf{$45^{\circ}$} &  220&(0)               & 211&(0)                 & 210&(0)                & 215&(0)                 \\
\textbf{$70^{\circ}$} &  221&(0)               & 202&(0)               & 199&(0)            &215&(0)                      \\
\textbf{$90^{\circ}$} &  206&(0)               & 199&(0)               & 200&(0)            &221&(0)                      \\
\bottomrule
\end{tabularx}
\end{table}

\subsection{Immiscible Three-Phase Gravity Segregation through Barriers} \label{sec:barriers}
Next, we consider a three-phase gravity segregation problem where the fluids travel through a field of barriers. This test case is taken from \cite{Klemetsdal2020Schwarz} with a couple modifications. A vertical domain of $100$m$\times100$m is gridded by $1$m$\times1$m cells. Initially, the top 10\% of cells are filled by a heavy fluid ($\rho=1500$kg/m$^3$), the bottom 10\% are saturated with a light fluid ($\rho=500$kg/m$^3$), and the remaining cells contain the intermediate fluid ($\rho=1000$kg/m$^3$). We name the fluids water(heavy), gas(light) and oil(intermediate), respectively, so that the problem is conceptually more familiar. All phases have viscosity $\mu=0.001 [kg/(m.s)]$ and quadratic relative permeability curves. Flow through the homogeneous rock, $k = 1$mD and $\phi =0.3$, is impeded by several impermeable layers. The openings in the layers create a complicated combination of gravitational and viscous dynamics with high and low flow regimes. Two versions of the test case are assessed: one without and one with capillary pressure. In the presence of capillary pressure, water is the most wetting and gas the least wetting phase. The spline interpolation presented in \ref{app:Capillary} (multiplied by 0.4) is applied for both the water-oil and oil-gas relationships. Figure \ref{fig:BarriersSnaps} depicts the initial condition and two time steps of the simulation. 

\begin{figure} [htbp]
\centering
\begin{subfigure}[t]{0.28\textwidth}
\centering
\includegraphics[width=\linewidth]{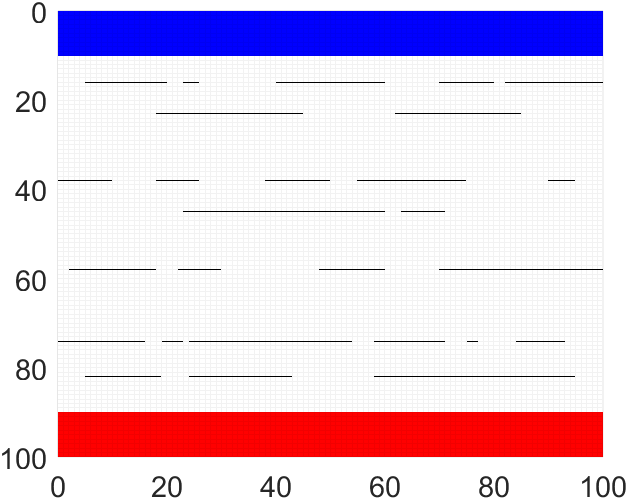}\hfill 
\caption{\label{fig:BarrierSnap1}Saturation map at 0 days}
\end{subfigure}
\hfill
\begin{subfigure}[t]{0.28\textwidth}
\centering
\includegraphics[width=\linewidth]{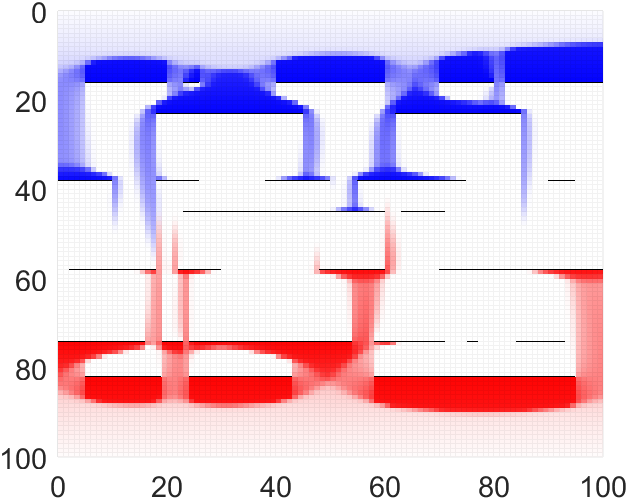}\hfill 
\caption{\label{fig:BarrierSnap2}Saturation map at 78 days}
\end{subfigure}
\hfill
\begin{subfigure}[t]{0.28\textwidth}
\centering
\includegraphics[width=\linewidth]{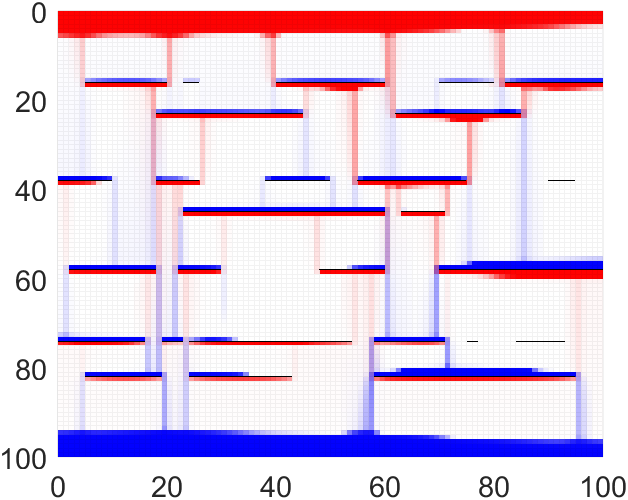}\hfill 
\caption{\label{fig:BarrierSnap3}Saturation map at 571 days}
\end{subfigure}
\hfill
\caption{\label{fig:BarriersSnaps} Snapshots of the saturation map for the simulation of gravity segregation through barriers without capillary forces. The light fluid (in red) is initially at the bottom and travels up while, simultaneously, the heavy fluid (in blue) is initially at the top and travels down. The intermediate fluid is indicated in white. }
\end{figure}

The cumulative Newton iterations are plotted in Figure \ref{fig:BarriersIterations}. We first discuss the results without capillary pressure in figure \ref{fig:BarriersIterations_nopc}. Although starting efficiently, the TM scheme runs into convergence issues, exceeding the maximum number of time step cuts at the 15$^{th}$ time step and stopping the simulation. The WA-HU TV scheme is the clear winner, requiring ~25\% and ~37\% less iterations than standard PPU and PPU-HU, respectively. We highlight that PPU-HU severely underperforms with respect to standard PPU in this test case. This is because PPU-HU can still suffer from the consequences of the same kinks as standard PPU, albeit to a lesser degree. Additionally it can also encounter issues associated to the kink at total velocity direction changes. This emphasizes the importance of smoothing the kinks at phase velocity direction changes, as WA-HU does, to fully take advantage of the benefits that fractional flow formulations and hybrid upwinding offer. 

The inclusion of capillary pressure smears the front and simplifies the problem for Newton's method in this test case. Nonetheless, Figure \ref{fig:BarriersIterations_withpc} still indicates a 25\% reduction in iterations from standard PPU to WA-HU TV. Furthermore, PPU-HU does not run into the same complications encountered without capillary pressure and here almost matches the most efficient performance. On the other hand, WA-HU TM still runs into severe convergence issues which stop the simulation. Overall, WA-HU TV again proves very robust and efficient.

\begin{figure}[htbp]
    \centering
    \begin{subfigure}[t]{0.49\textwidth}
    \centering
    \includegraphics[width=\linewidth]{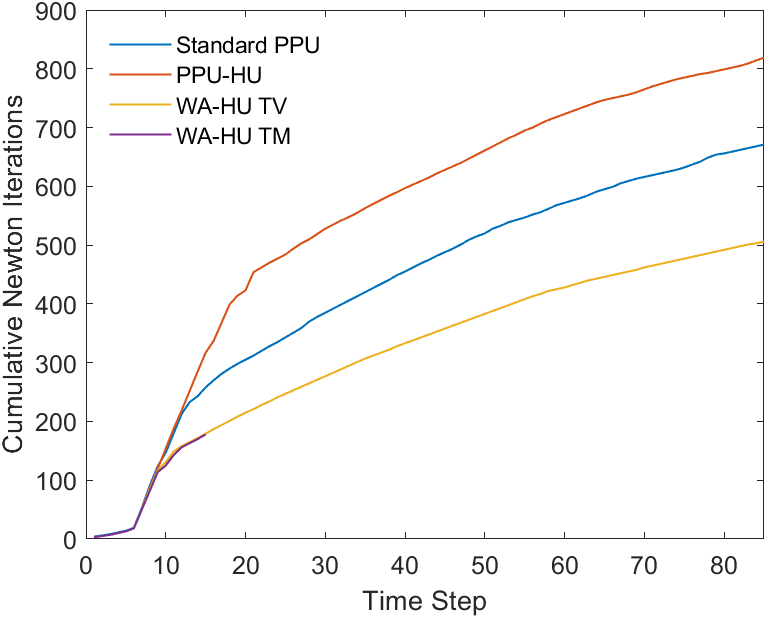}\hfill 
    \caption{\label{fig:BarriersIterations_nopc}}
    \end{subfigure}
    \hfill
    \begin{subfigure}[t]{0.49\textwidth}
    \centering
    \includegraphics[width=\linewidth]{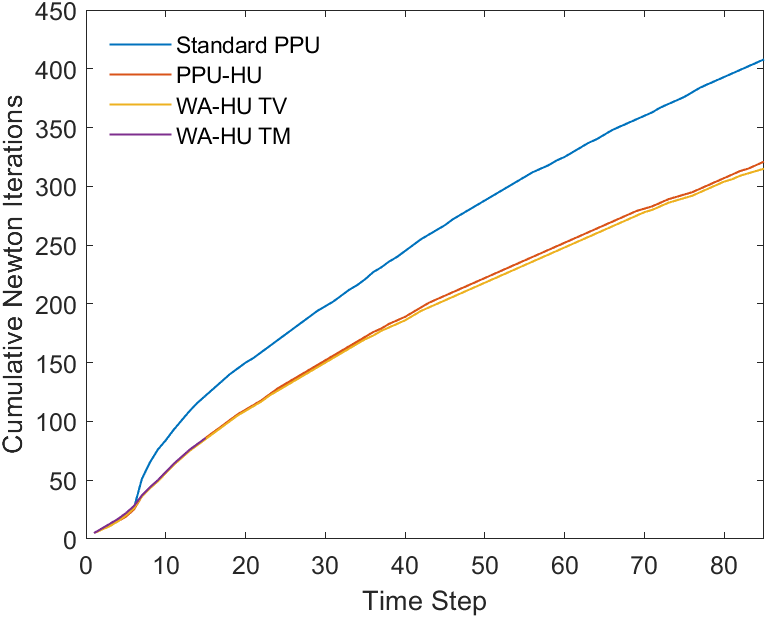}\hfill 
    \caption{\label{fig:BarriersIterations_withpc}}
    \end{subfigure}
    \hfill
    \centering
    \caption{Cumulative Newton iterations for the three-phase barriers test case (a) without capillary forces and (b) with capillary forces. Time step sizes are equivalent for all schemes. The cumulative counts include wasted iterations and intermediate time step iterations due to time step cuts. }
    \label{fig:BarriersIterations}
\end{figure}

\subsection{Miscible CO$_2$ Plume Migration} \label{sec:Plume}
The next test cases are run with more realistic fluids. We first consider the injection of CO$_2$ into a brine reservoir. Here, the rock is homogeneous with permeability in x- and z-direction equal to $k_x=100$mD and $k_z=1$mD, and porosity $\phi=0.3$. The reservoir is $100$m$\times10$m$\times10$m in size. We discretize the domain into $50\times1\times20$ cells for a coarse grid, and into $200\times1\times80$ cells for a fine grid. We take fluid properties for water($w$) and CO$_2$($g$) from CoolProp \cite{Bell2014Coolprop} via MRST. Relative permeability curves are set to $k_{r,w} = S_w^{2}$ and $k_{r,g} = S_g^2$. The top right cell is set to a fixed pressure of 50 atmospheric bars and we inject CO$_2$ at a constant rate in the bottom left such that 0.5 pore volumes are injected by the end of the simulation. We ramp up to $\Delta t =100$days in eight time steps and finish the simulation at 5 years. An illustration of the CO$_2$ plume is shown in figure \ref{fig:CO2Snap}.
\begin{figure}[htbp]
    \centering
    \includegraphics[width=.7\textwidth]{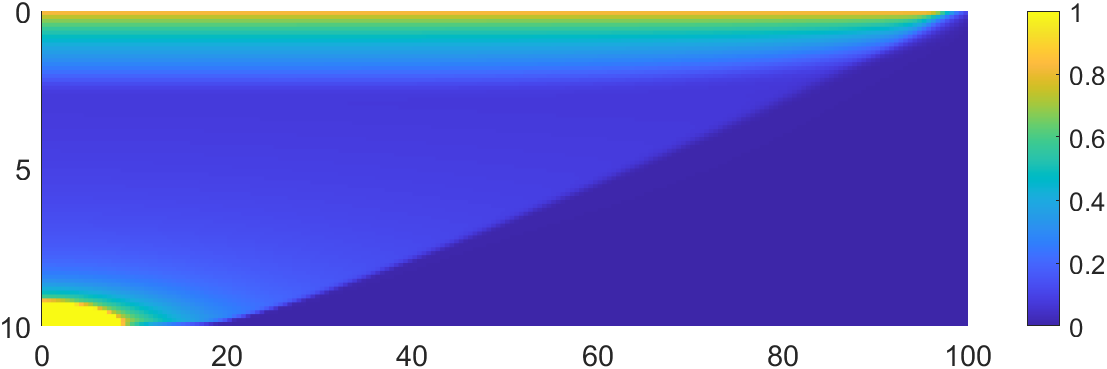}
    \caption{Overall CO$_2$ mass fraction on the fine grid.}
    \label{fig:CO2Snap}
\end{figure}

Although not a complicated test case, the results validate that the new proposed schemes perform adequately for less challenging problems. As figure \ref{fig:CO2Iterations_fine} specifies, the number of iterations does not differ greatly between the tested schemes for the coarse grid simulation.  Moreover, unlike in the other test cases, we note that the small differences in performance seen here are not significant enough to warrant a conclusion. When the grid is refined, Newton with a standard PPU scheme starts to struggle (see figure \ref{fig:CO2Iterations_coarse}). Most additional iterations occur towards the end of the simulation. This is most likely due to the gravity dynamics induced by heavier water with dissolved CO$_2$ at the top of the domain.
\begin{figure}[htbp]
    \centering
    \begin{subfigure}[t]{0.48\textwidth}
    \centering
    \includegraphics[width=\linewidth]{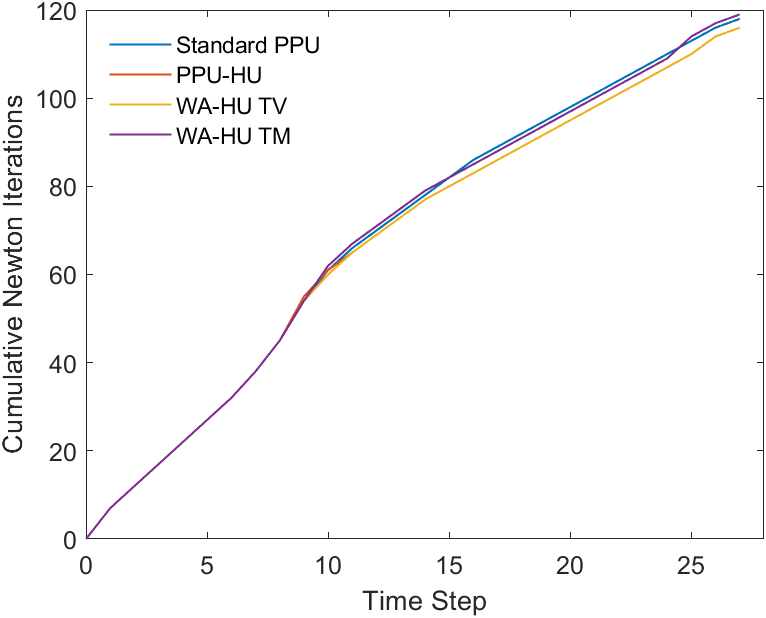}
    \caption{\label{fig:CO2Iterations_coarse}Coarse grid simulation.}
    \end{subfigure}
    \hfill
    \begin{subfigure}[t]{0.48\textwidth}
    \centering
    \includegraphics[width=\linewidth]{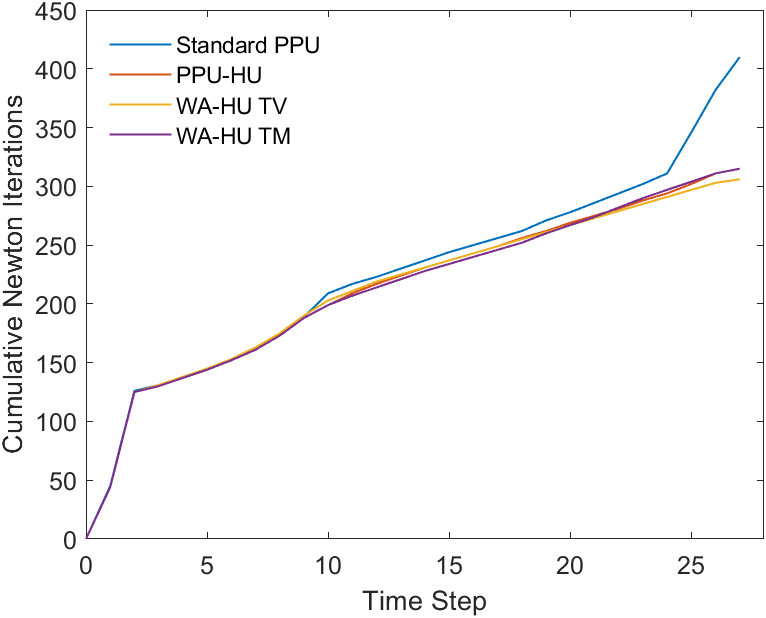}
    \caption{\label{fig:CO2Iterations_fine}Fine grid simulation.}
    \end{subfigure}
    \hfill
    \caption{Cumulative Newton iterations for the CO$_2$ plume test case. Time step sizes are equivalent for all schemes. The cumulative counts include wasted iterations and intermediate time step iterations due to time step cuts. Due to overlapping lines, we specify that PPU-HU requires the same number of iterations as WA-HU TV in the coarse grid simulation (a), and performs equivalently to WA-HU TM in the fine simulation (b).  }
    \label{fig:CO2Iterations}
\end{figure}

\subsection{Compositional Heterogeneous Flow through High Permeability Conduits} \label{sec:fractures}
This test case originates from \cite{Klemetsdal2020Schwarz} and \cite{Moyner2018Sequential} (see former for additional details). We alter it slightly to make it more challenging. A $100$m$\times50$m reservoir has five heterogeneous layers which are intersected by 13 conductive fracture channels. We consider three versions of the test case, a high contrast, a low contrast version and a low contrast version with capillary pressure. With high contrast the fractures are two orders of magnitude more conductive than the most conductive matrix cells. In the low contrast case, the fractures and most conductive matrix cells have permeabilities in the same order of magnitude. The Voronoi-cell grid and high contrast heterogeneous permeability field are illustrated in Figure \ref{fig:FracturesPerm}. The reservoir contains a two-phase, liquid-gas mixture consisting of n-decane, carbon dioxide and methane. The phase behavior is described by the cubic Peng-Robinson equation-of-state. A mixture of n-decane and carbon dioxide is injected in the bottom left of the domain at a fixed rate. For the case with capillary pressure, the liquid is considered wetting and we use the relationship from \ref{app:Capillary} scaled by a factor of $0.4$.  A bottom hole pressure well produces fluids at 50bar in the top right corner. The full list of the unaltered test case parameters can be found in the examples suite of the development version of MRST \cite{MRST}. 

\begin{figure}[htbp]
    \centering
    \includegraphics[width=.5\textwidth]{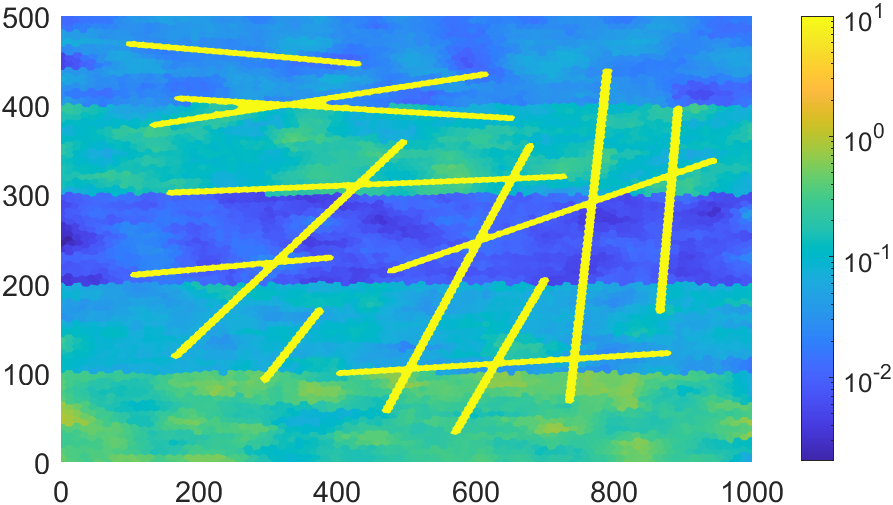}
    \caption{Permeability [Darcy] map of the high contrast fractures test case. }
    \label{fig:FracturesPerm}
\end{figure}

To make the case more challenging than the original, we inject at a faster rate, namely 0.3 pore volumes in 3 years. We also take larger time steps; ramping up to 20 (high contrast and low contrast with capillary pressure) and 25 days (low contrast) after eight initial smaller steps. The gas saturation and CO$_2$ mass fraction approximately halfway the simulation are presented in figure \ref{fig:FracturesSnaps}. We reiterate that these maps are close to identical for all schemes and that in the limit they converge to the same solution. As the test case becomes significantly harder than previous test cases, we apply Modified Appleyard saturation update damping with a maximum update of 0.2 (also see Section \ref{sec:Newton}). This also allows us to compare the new schemes in industry standard simulation settings.

\begin{figure} [htbp]
\centering
\begin{subfigure}[t]{0.48\textwidth}
\centering
\includegraphics[width=\linewidth]{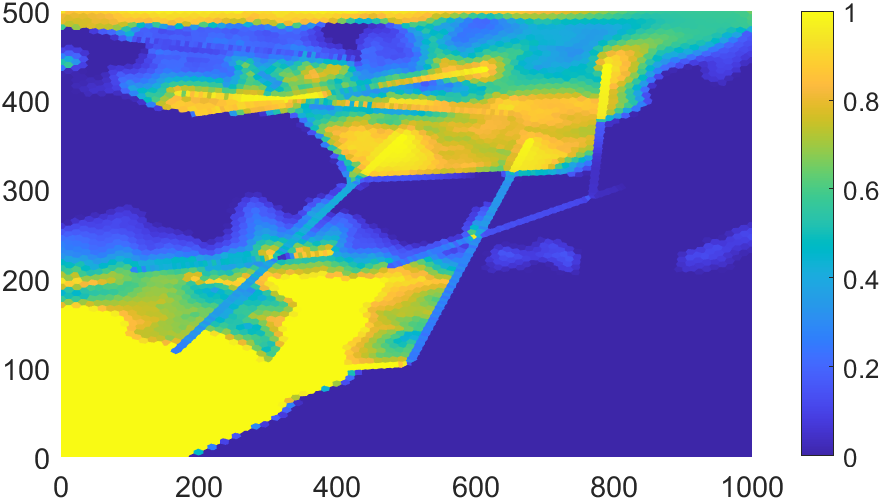}\hfill 
\caption{\label{fig:FracturesSgHighContrast}Gas saturation at 640 days.}
\end{subfigure}
\hfill
\begin{subfigure}[t]{0.48\textwidth}
\centering
\includegraphics[width=\linewidth]{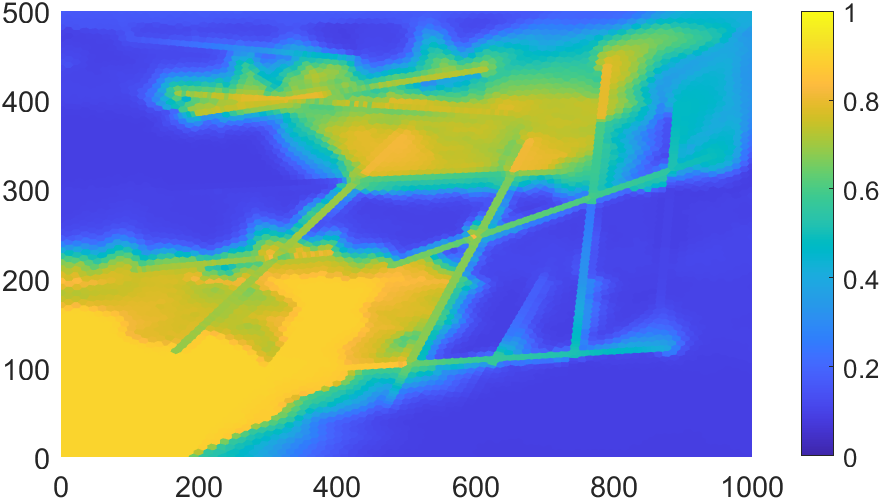}\hfill 
\caption{\label{fig:FracturesxCO2HighContrast}CO$_2$ mole fraction at 640 days}
\end{subfigure}
\hfill
\centering
\begin{subfigure}[t]{0.48\textwidth}
\centering
\includegraphics[width=\linewidth]{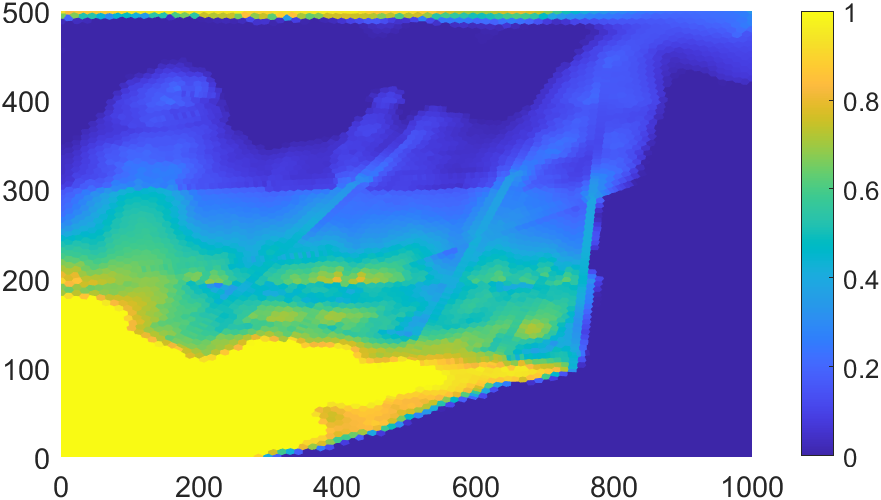}\hfill 
\caption{\label{fig:FracturesSg}Gas saturation at 600 days.}
\end{subfigure}
\hfill
\begin{subfigure}[t]{0.48\textwidth}
\centering
\includegraphics[width=\linewidth]{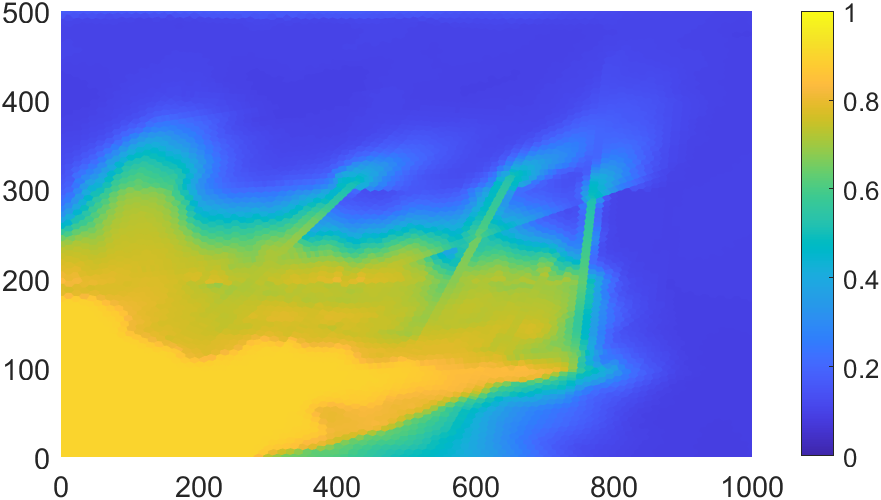}\hfill 
\caption{\label{fig:FracturesxCO2}CO$_2$ mole fraction at 600 days}
\end{subfigure}
\hfill
\caption{\label{fig:FracturesSnaps} Gas saturation and CO$_2$ mass fraction after 640 days for the fracture test cases without capillary pressure. Figures (a) and (b) are from the high contrast simulation, while (c) and (d) are from the low contrast simulation.}
\end{figure}

The cumulative iteration counts in figure \ref{fig:FracturesIterations} show that the WA-HU schemes outperform standard PPU and PPU-HU consistently. Looking at the low contrast case, figure \ref{fig:FracturesIterationsLow}, the new schemes require fewer iterations in time steps which are easier (e.g. 30 to 40). The difference is more significant for time steps with complicated dynamics, e.g. time steps 25 to 28, where the cumulative Newton iteration count increases significantly less for WA-HU because of fewer time step cuts. Overall, the new proposed schemes provide a reduction in nonlinear iterations of more than 35\% over standard PPU and more than 27\% over PPU-HU. In the high contrast case, Figure \ref{fig:FracturesIterationsHigh}, the difference in performance is more significant. Especially standard PPU suffers, leading to reductions in Newton iterations of 45\% and almost 60\% for the WA-HU TV and WA-HU TM schemes, respectively. Finally, in the case with capillary pressure we see that the results are more similar. This is in part due to the severe complexity of the problem which leads to a lot of time steps for all schemes. Still the new schemes significantly improve Newton performance, where again the WA-HU TM allows for the simulation to run with the least Newton iterations. It is interesting to point out that although the maximum MRST\cite{MRST} compositional CFL numbers are similar for all cases (approximately 150), the maximum saturation CFL number for the high contrast case is almost seven times higher (440 vs. $\sim 65$). This is in line with earlier works on HU, which show that the benefits increase as the saturation volume exchanges increase in the simulation. Summing up, this test case shows that WA-HU TM can increase efficiency even further for compositional simulations. Both new schemes again show strong performance across all variations. 

\begin{figure}[htbp]
    \centering
    \begin{subfigure}[t]{0.49\textwidth}
    \centering
    \includegraphics[width=\linewidth]{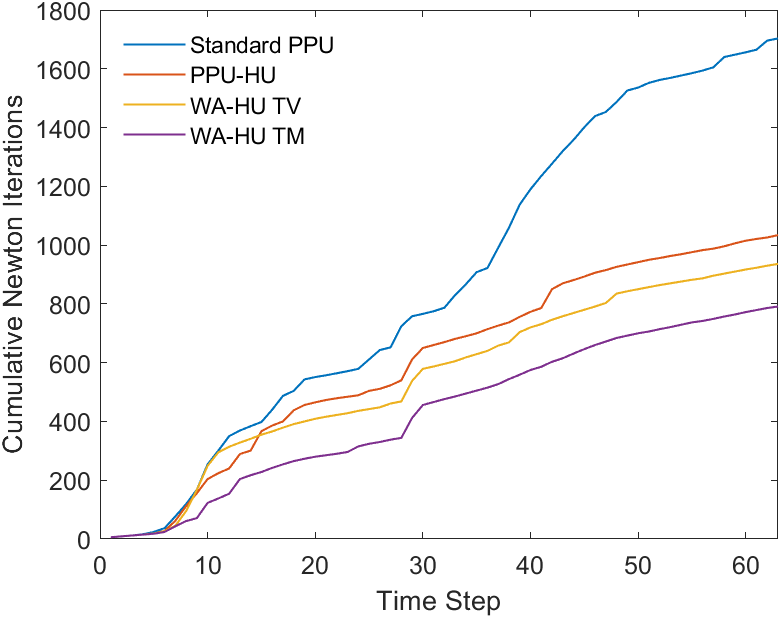}\hfill 
    \caption{\label{fig:FracturesIterationsHigh}High contrast case}
    \end{subfigure}
    \hfill
    \begin{subfigure}[t]{0.49\textwidth}
    \centering
    \includegraphics[width=\linewidth]{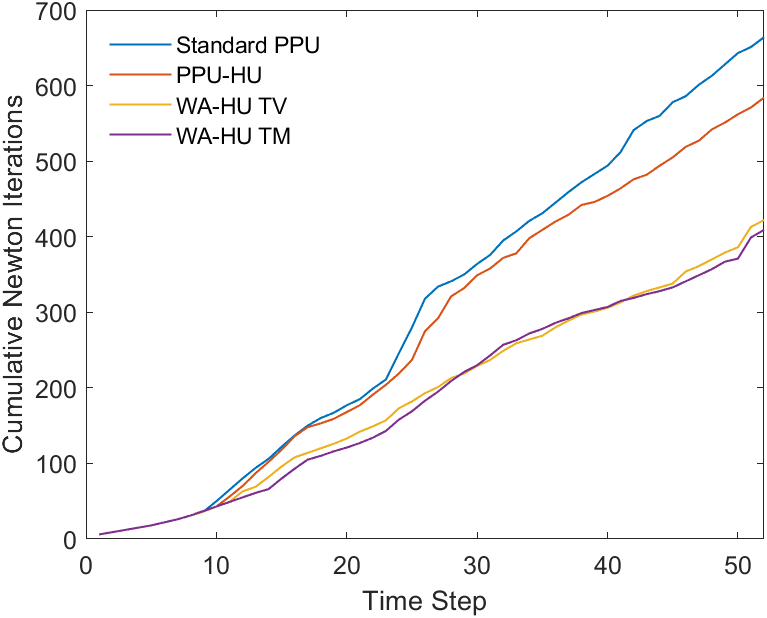}\hfill 
    \caption{\label{fig:FracturesIterationsLow}Low contrast case}
    \end{subfigure}
        \begin{subfigure}[t]{0.49\textwidth}
    \centering
    \includegraphics[width=\linewidth]{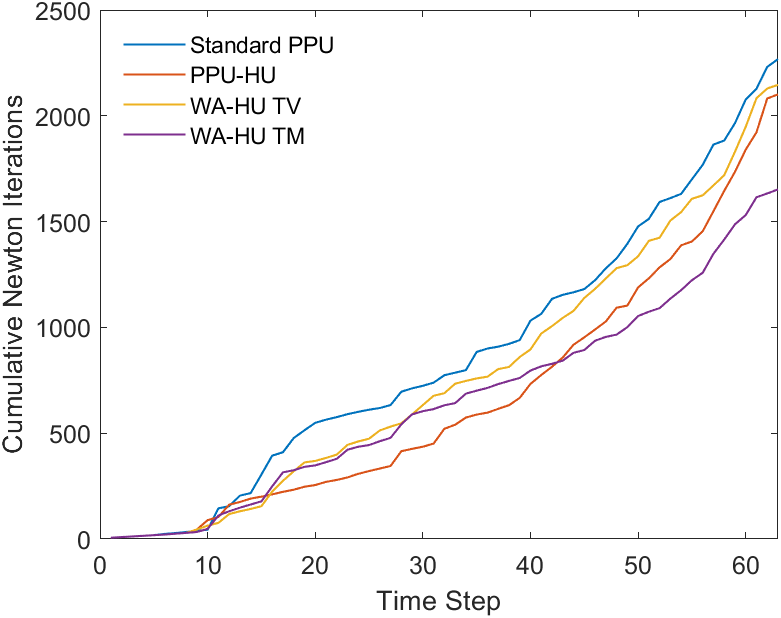}\hfill 
    \caption{\label{fig:FracturesIterationsLowPc}Low contrast case with capillary forces}
    \end{subfigure}
    \caption{Cumulative Newton iterations for the fracture test case with (a) high, (b) low permeability contrast, and (c) with capillary forces. Time step sizes are equivalent for all schemes. The cumulative counts include wasted iterations and intermediate time step iterations due to time step cuts.}
    \label{fig:FracturesIterations}
\end{figure}

\subsection{Water Alternating Gas Injection into SPE10 Layer} \label{sec:SPE10}
Finally, we consider a hydrocarbon production case using water-alternating-gas (WAG) injection. Practically, the alternation of the injection fluids upholds a favorable mobility ratio; numerically, it leads to a challenging simulation problem. The test case is a variation of the example presented in \cite{Moncorge2018SFIcompositional}. The first layer of the SPE10 Model 2 is initially filled with a six-component (C1, C3, C6, C10, C15, and C20) liquid-mixture. We rotate the reservoir and assume it is tilted at an angle such that gravity acts along the original x-axis (our y-axis) with an acceleration of $g=4.9$m/s$^2$. Figure \ref{fig:SPE10Perm} shows the rotated permeability field. A well in the southwest corner injects C1 gas for 1250 days, followed by immiscible water for 1250 days, followed by C1 gas for another 1250 days. A well in the northeast corner produces at a fixed bottom-hole pressure of 275 bar. Time steps of 75 days are taken after ramping up for the initial eight time steps of each injection phase. Fluid details of the three-phase system can be found in the development version of MRST \cite{MRST}. The final methane mole fraction is presented in Figure \ref{fig:SPE10CH4}.

\begin{figure} [htbp]
\centering
\begin{subfigure}[t]{0.48\textwidth}
\centering
\includegraphics[width=\linewidth]{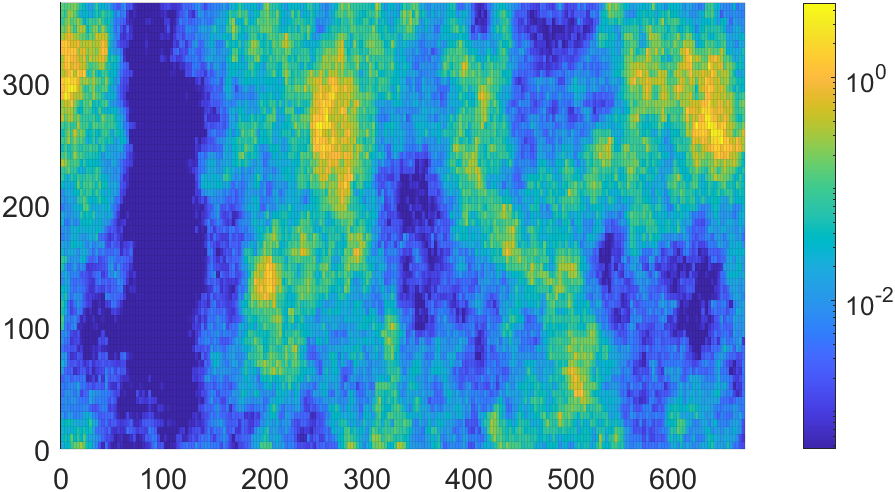}\hfill 
\caption{\label{fig:SPE10Perm}Permeability[Darcy]}
\end{subfigure}
\hfill
\begin{subfigure}[t]{0.48\textwidth}
\centering
\includegraphics[width=\linewidth]{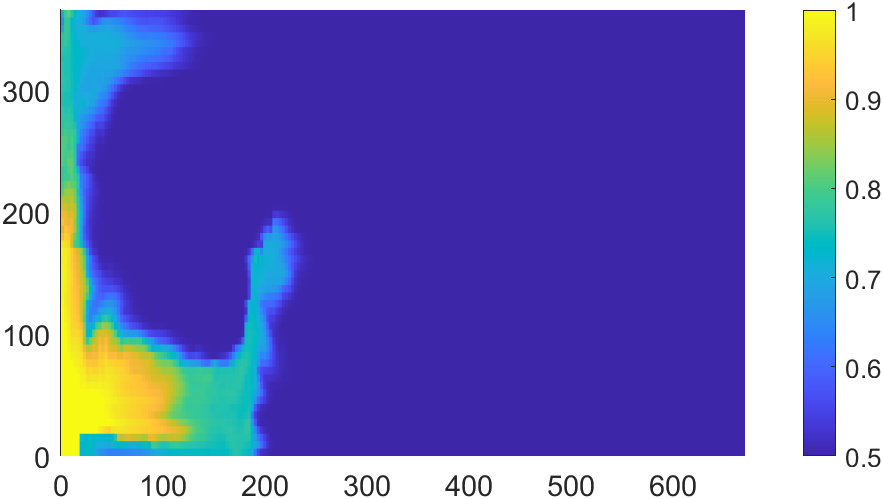}\hfill 
\caption{\label{fig:SPE10CH4}CH$_4$ mole fraction at end of simulation}
\end{subfigure}
\hfill
\centering
\caption{\label{fig:SPE10Snaps} SPE10 WAG test case permeability field and methane mole fraction map. Note that the test fluids are initialized at a methane mole fraction of 0.5.}
\end{figure}

The cumulative Newton iteration counts are plotted in Figure \ref{fig:SPE10Iterations}. Although through the ramp up time steps (0-8, 25-33 and 50-58) the schemes perform similarly, the difference between standard PPU and the three HU formulations is pronounced for all large time steps. Overall, this leads to a 21-26\% decrease in nonlinear iterations. Between PPU-HU, WA HU TV and WA HU TM the differences are not large enough to warrant conclusions. We again confirm the excellent performance of the proposed schemes.

\begin{figure}[htbp]
    \centering
    \includegraphics[width=.5\textwidth]{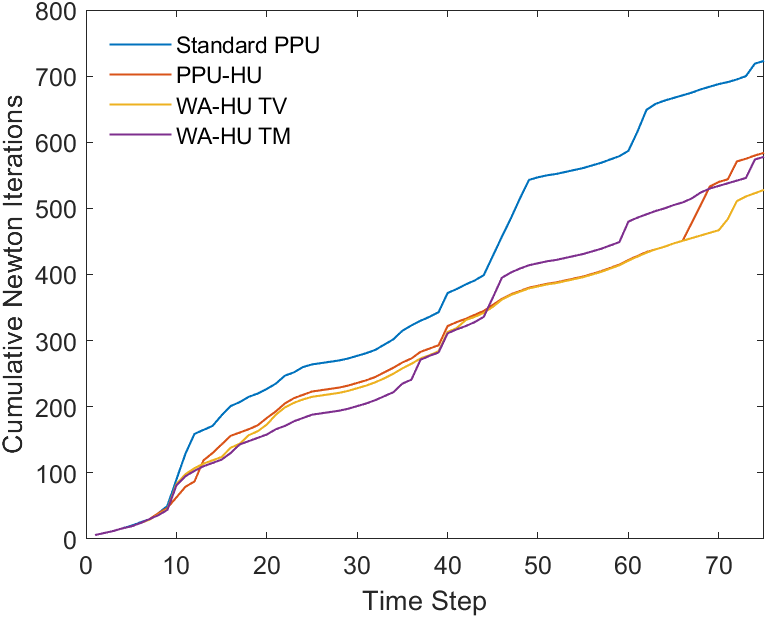}
    \caption{Cumulative Newton iterations for the SPE10 test case. Time step sizes are equivalent for all schemes. The cumulative counts include wasted iterations and intermediate time step iterations due to time step cuts. }
    \label{fig:SPE10Iterations}
\end{figure}

\section{Conclusion} \label{sec:conclusion}

Efficient simulations capable of taking large time steps are essential in managing hydrocarbon development and long-term geological carbon sequestration. To this end, we proposed two discretization schemes, WA-HU TV and WA-HU TM, which guide the iterations of the Newton-Raphson nonlinear solver more effectively, and thus allow for quicker convergence. By appropriately treating the fluid properties for each of the driving forces, we show that we can avoid the majority of numerics-induced kinks in the solution space. This is illustrated and emphasized through a thorough analysis of the residual contours and continuous Newton paths of a one-cell problem.

The schemes' performance vis-a-vis the industry standard method (PPU) and current state-of-the-art method (PPU-HU) is assessed on a suite of test cases. The problems range from simple 1D set-ups, to synthetic problems designed to challenge the nonlinear solver, to realistic compositional reservoir processes. The proposed total velocity hybrid upwinding scheme with weighted averaged properties (WA-HU TV) consistently outperforms the existing alternatives, yielding benefits from 5\% to over 50\% reduction in nonlinear iterations. The improved performance is most significant for challenging and realistic test cases. Overall, based on the current results, we recommend the adoption of the WA-HU TV scheme as it is highly efficient and robust.

The WA-HU scheme in a total mass flux formulation (WA-HU TM) also showed promising results; in some cases leading to even more efficiency. However, more work needs to be done on understanding the Newton behavior related to the total mass formulation and how to resolve the encountered issues. Several alternatives were explored in the process of this work. In short, the two most successful strategies chopped updates based on (i) the percentage of interfaces which had switched signs and (ii) on the number of flip-flops (traverses of the kink) an individual interface had experienced in a certain time step. Although severe time stepping issues can be avoided, these damping methods also lead to increased iterations in simulations where damping is not necessary. More generally, tuning these strategies is critical to their success. As such, a more robust, information-based and generally applicable chopping strategy to resolve convergence issues, such as with WA-HU TM, is still an open topic of interest.  

Summing up this work's contributions, we proposed (i) an improved WA treatment for flow mobilities in the fractional flow formulation, (ii) an extension of the HU framework to full-physics flow and transport simulations with compositional fluids, and (iii) a novel fractional flow formulation based on total mass flux. The schemes are extensively tested on a variety of cases and show significant nonlinear performance improvements across the board.


\subsection{Translating Nonlinear Performance to Run Time Savings}
In this work, we evaluate performance based on Newton iteration counts. Ultimately, we are of course interested in simulation run time. As a highly efficient implementation is a study in itself, it is not a focus of this work. Nevertheless, we note that our HU-based linear system construction (coded for research purposes, not optimized) costs $\sim$20\% more than standard system construction (open source optimized code). This is in line with the theoretical estimate that HU formulations should cost approximately 10\% more than standard constructions. We highlight that this additional cost is notably less than the savings in non-linear iterations. Additionally, as hardware becomes more powerful, the cost of linear system assembly is becoming increasingly less relevant compared to the rest of the iteration cost (such as linear system solves). As such, we plan to utilize WA-HU TV in our future work for general purpose simulation.

\section*{Acknowledgements} 
SB is supported by a named Stanford Graduate Fellowship in Science and Engineering (SGF). 
Funding for FH was provided by TotalEnergies through the FC-MAELSTROM project. FH contributed to this work during a visiting scientist appointment at LLNL.
The authors thank the SUPRI-B research group for valuable discussions and feedback. 
The authors thank {\O}ystein Klemetsdal for his help with MRST.

\appendix 

\section{Proof of Monotonicity For Incompressible Flow and Transport}
\label{app:Proof}
Almost all the properties proven in \cite{Hamon2016HU_Buoyancy} still apply to the new WA-HU scheme. To complete the required properties for the new discretization scheme, this appendix presents the proof of monotonicity of the total velocity with respect to pressure. As the requirement of monotonicity only applies to incompressible immiscible systems, those assumptions on the fluid properties are applied for the purpose of the proof.  

As the HU scheme with PPU flow mobilities is monotone, we can state that WA-HU is monotone if 
\begin{align} \label{eq:monotonicityConstraint}
    \frac{\partial u_{t}}{\partial \Delta p_{\ell,ij}} \geq 0.
\end{align}
Hence, to prove monotonicity, we prove a sufficient condition on $\gamma_{\ell,ij}$ for eq. \eqref{eq:monotonicityConstraint} to hold. From here on, we omit the face subscripts for notation simplicity.

\begin{lemma}
Considering $\beta_{\ell}$ as defined in eq. \eqref{eq:beta}, WA-HU satisfies eq. \eqref{eq:monotonicityConstraint} if
\begin{align}
    \gamma_\ell \geq 0.
\end{align}
\end{lemma}

\begin{proof}
Starting from eq. \eqref{eq:totalvelocity} with weighted average mobilities, we take the partial derivative with respect to the phase pressure difference:
\begin{align}
    \frac{\partial u_t}{\partial \Delta p_\ell} = T \sum_\ell \Bigg[ \frac{\partial \lambda^{WA}_\ell}{\partial \Delta p_\ell} \Delta \Phi_\ell + \lambda^{WA}_\ell \Bigg]
\end{align}
Filling in eq. \eqref{eq:weightedAverageMobility} and rewriting, we obtain

\begin{align}
    \frac{\partial u_t}{\partial \Delta p_\ell} = T \sum_\ell \Bigg[ \bigg(\beta_\ell + \frac{\partial \beta_\ell}{\partial \Delta p_\ell} \Phi_\ell \bigg) (\lambda_{\ell,i} - \lambda_{\ell,j}) + \lambda_{\ell,j} \Bigg].
\end{align}
If we can prove that each of the components of the sum satisfies eq. \eqref{eq:monotonicityConstraint}, then the total will as well. To this end, we can state that if
\begin{align} \label{eq:toprove}
    0 \leq \bigg(\beta_\ell + \frac{\partial \beta_\ell}{\partial \Delta p_\ell} \Phi_\ell \bigg) \leq 1,
\end{align}
then the scheme is monotone. 

We show that eq. \eqref{eq:toprove} holds in 3 steps: first we compute the 2 limits of the term, and then we evaluate if it is non-decreasing throughout. Substituting $\Gamma_\ell = \frac{\gamma_\ell}{|\boldsymbol{g}_{ref}| + |\boldsymbol{c}_{\ell,ref}|}$, we rewrite eq. \eqref{eq:beta} to be
\begin{align} \label{eq:beta2}
    \beta_{\ell} = 0.5 + \frac{1}{\pi} \arctan\big(\Gamma_\ell \Delta \Phi_{\ell}\big).
\end{align}
Of which we then find the partial derivative 
\begin{align} \label{eq:betaderivative}
    \frac{\partial \beta_{\ell}}{\partial \Delta p_\ell} = \frac{\Gamma_\ell}{\pi \big(1 +(\Gamma_\ell \Delta \Phi_{\ell})^2\big)}.
\end{align}
The term in eq. \eqref{eq:toprove} is a differentiable function of $\Delta p_\ell$, to which we assign the function variable $P$. Inserting eq. \eqref{eq:beta2} and eq. \eqref{eq:betaderivative}, we obtain
\begin{align} \label{eq:termP}
    P(\Delta p_\ell) = 0.5 + \frac{1}{\pi} \arctan\big(\Gamma_\ell \Delta \Phi_{\ell}\big) + \frac{\Gamma_\ell \Delta \Phi_\ell}{\pi \big(1 +(\Gamma_\ell \Delta \Phi_{\ell})^2\big)}.
\end{align}
We can then evaluate the bounds of the function and conclude:
\begin{align}
    \lim_{\Delta p_\ell \rightarrow -\infty} P(\Delta p_\ell) = 0 \quad \text{and} \quad \lim_{\Delta p_\ell \rightarrow \infty} P(\Delta p_\ell) = 1.
\end{align}
To determine if the term is non decreasing, we again take the derivative. This yields
\begin{align}
    \frac{\partial P}{\partial \Delta p_\ell} & = \frac{\Gamma_\ell}{\pi \big(1 +(\Gamma_\ell \Delta \Phi_{\ell})^2\big)} +
    \frac{\pi \Gamma_\ell \big(1 +(\Gamma_\ell \Delta \Phi_{\ell})^2\big) -
    2 \pi \Gamma_\ell^3 (\Delta \Phi_{\ell})^2}
    {\pi^2 \big(1 +(\Gamma_\ell \Delta \Phi_{\ell})\big)^2} \nonumber \\
    & = \frac{\Gamma_\ell}{\pi \big(1 +(\Gamma_\ell \Delta \Phi_{\ell})^2\big)}
    \bigg[ 1+ \frac{1-(\Gamma_\ell \Delta \Phi_{\ell})^2}{1+(\Gamma_\ell \Delta \Phi_{\ell})^2}\bigg].
\end{align}
Here, the second term and the denominator of the first term are strictly larger than 0. So for
\begin{align}
    \frac{\partial P}{\partial \Delta p_\ell} \geq 0
\end{align}
to be guaranteed, it suffices that $\Gamma_\ell \geq 0$, or more simply $\gamma_\ell \geq 0$. Hence, if $\gamma_\ell \geq 0$, then \eqref{eq:toprove} is satisfied, and thus \eqref{eq:monotonicityConstraint} holds.
\end{proof}

\section{Capillary Relationship Data and Good Practices}
\label{app:Capillary}
This appendix comments on good practices for efficient simulations with capillary input functions and presents the capillary data used in this work.

\subsection{Capillary functions for efficient simulation}
Many popular capillary relationships, such as the Brooks-Corey model, have an asymptote as the wetting effective saturation approaches zero. This asymptotic behavior can have significant consequences on the shape of the total velocity profile and residual space while having little impact on the solution. As such it also often leads to unnecessary additional Newton iterations or convergence issues. To illustrate the effect, figure \ref{fig:TotalVelocityInf} plots the total velocity profiles using the analytical capillary function related to the input data (see next subsection). Figures \ref{fig:TotalVelocityPPUInf} and \ref{fig:TotalVelocityWAnewInf} clearly show the steep asymptote which was not present when using spline interpolation in Figures \ref{fig:TotalVelocityPPU} and \ref{fig:TotalVelocityWAnew}. It is also evident that the rest of the domain remains close to identical in shape. The original WA scheme with capillary forces \cite{Hamon2018HU_capillarity} explicitly aimed to remove this asymptote through the discretization. Rather than treating capillary forces in the total velocity, it separately treats the forces in a fully smooth term. Although effective at smoothing, this results in additional numerical dispersion. The total velocity profile is depicted in figure \ref{fig:TotalVelocityWAoldInf}. To complete the analysis we also combine the capillary treatment of the original scheme with the new weighting in figure \ref{fig:TotalVelocityWAmixInf}. As is depicted in the plots, the asymptote can effectively be avoided at the expense of a slightly different total velocity profile. 

\begin{figure} [htbp]
\centering
\begin{subfigure}[t]{0.45\textwidth}
\centering
\includegraphics[width=\linewidth]{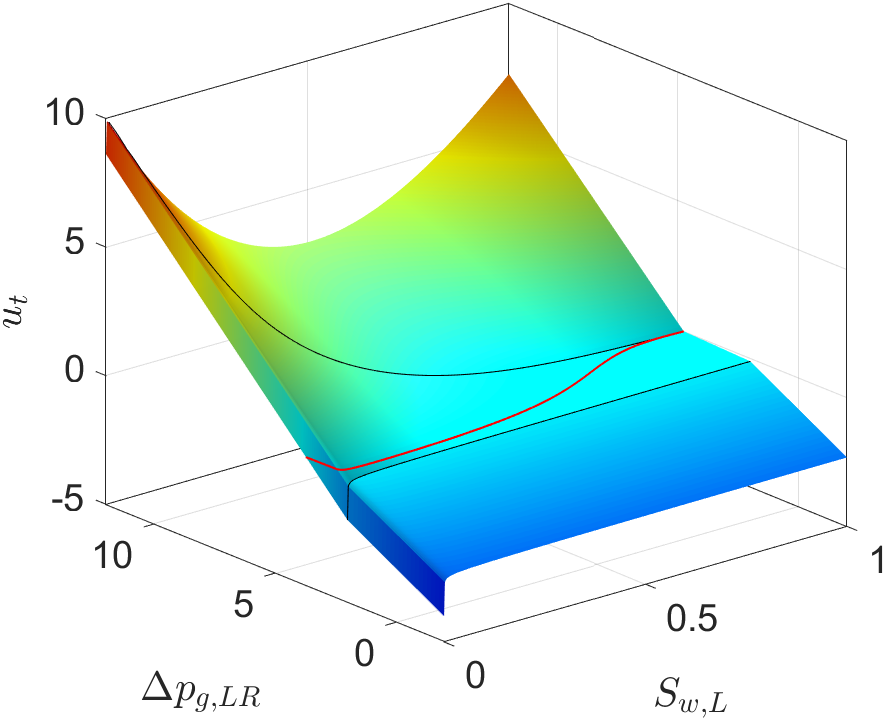}\hfill 
\caption{\label{fig:TotalVelocityPPUInf} PPU}
\end{subfigure}
\hfill
\begin{subfigure}[t]{0.45\textwidth}
\centering
\includegraphics[width=\linewidth]{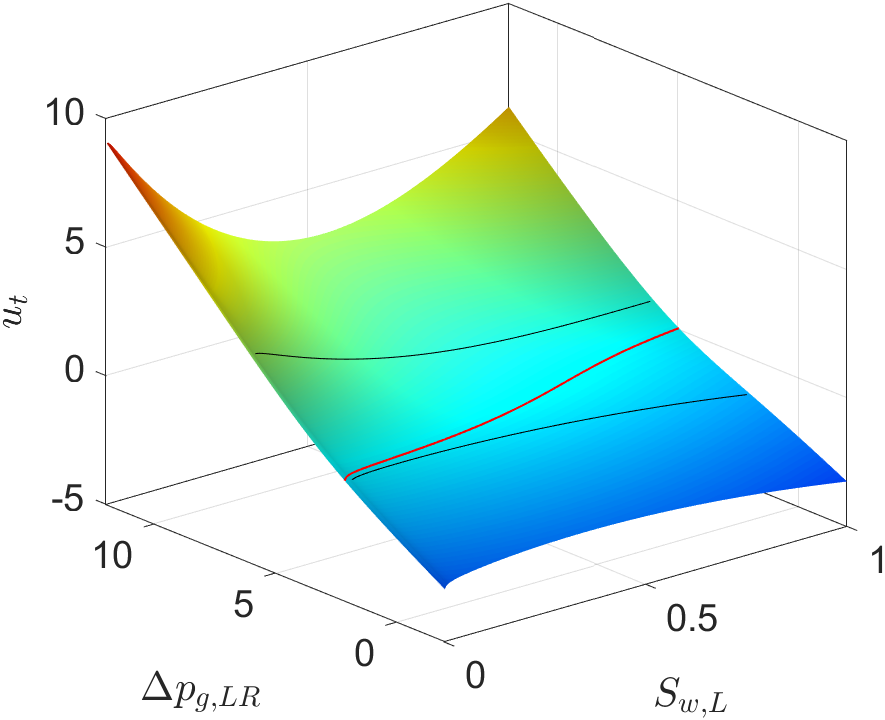}\hfill 
\caption{\label{fig:TotalVelocityWAoldInf}WA (Hamon et al. \cite{Hamon2018HU_capillarity})}
\end{subfigure}
\hfill
\begin{subfigure}[t]{0.45\textwidth}
\centering
\includegraphics[width=\linewidth]{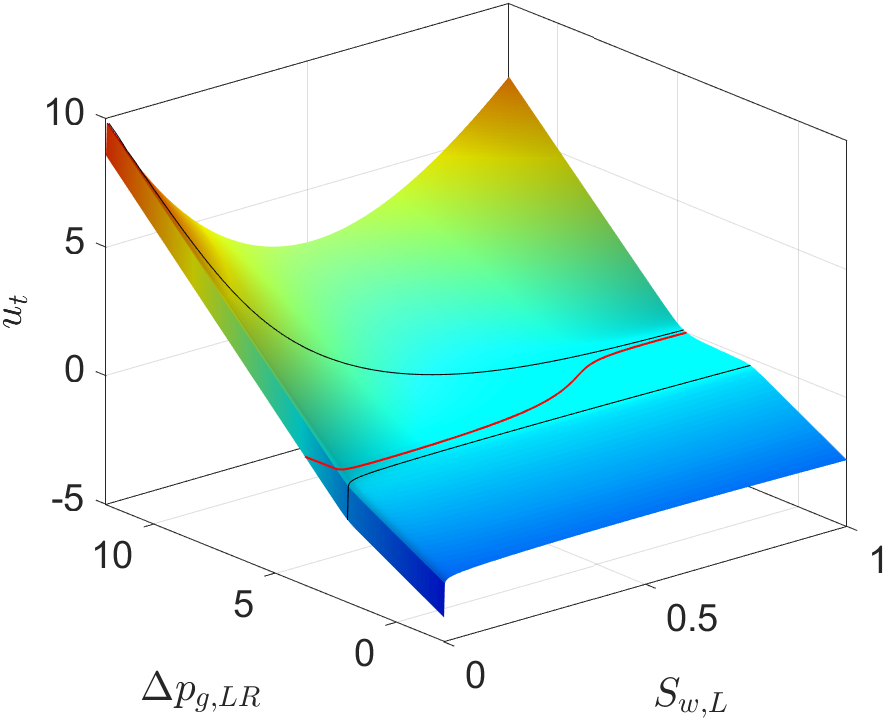}\hfill 
\caption{\label{fig:TotalVelocityWAnewInf}WA}
\end{subfigure}
\hfill
\begin{subfigure}[t]{0.45\textwidth}
\centering
\includegraphics[width=\linewidth]{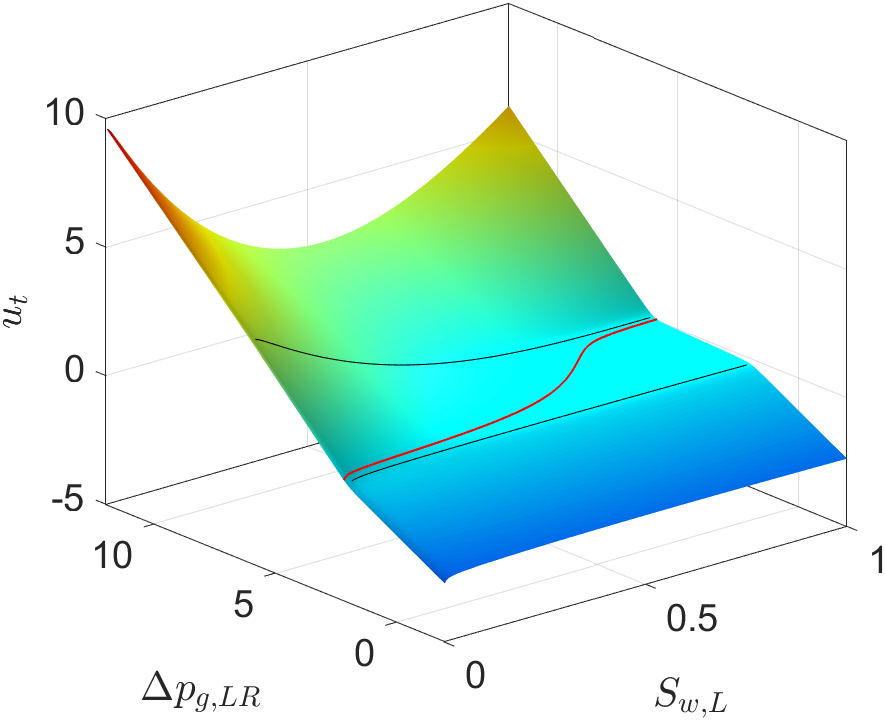}\hfill 
\caption{\label{fig:TotalVelocityWAmixInf}WA mix}
\end{subfigure}
\hfill
\caption{\label{fig:TotalVelocityInf} Total velocity profiles of the one-cell proxy test case for the PPU, WA \cite{Hamon2018HU_capillarity}, new WA scheme and a mix WA scheme. The latter uses the capillary treatment from \cite{Hamon2018HU_capillarity} and the weights from the new WA scheme. The left cell pressure and saturation are varied, while the right cell boundary conditions remain fixed. Phase potential flips and total velocity flips (i.e. changes in direction) are indicated with black and red lines, respectively. Note that in the original WA scheme and WA mix scheme, capillary pressures are excluded from the phase potential flips because they are treated separately. }
\end{figure}

Although this strategy is effective, it is essentially resolving an input issue. As such, in this work we opt for the discretization with the best properties, and avoid input functions with asymptotes. Moreover the new WA scheme is simpler in its formulation and implementation, and slightly less computationally costly per iteration. Recalling that the inclusion of the asymptote in the capillary pressure affects the numerical result very little, whereas nonlinear performance suffers greatly, we recommend using input functions for capillary pressure without asymptotes to infinity.

\subsubsection{Capillary input data}
In this work we use a capillary input data table. The data points are taken from the Brooks-Corey relationship with the following quantities:
\begin{align}
    P_{cap} = 5e4 \ \Bigg(\frac{1}{S_{wet}}\Bigg)^4,
\end{align}
where $S_{wet}$ is the wetting phase saturation and we ignore residual saturations. To avoid the asymptote present in the analytical function, we use a cubic spline interpolation on the table data. The table input is presented in Table \ref{tab:capData} and the interpolation in \ref{fig:capCurve}.

\begin{figure}[htbp]
    \centering
    \includegraphics[width=.5\textwidth]{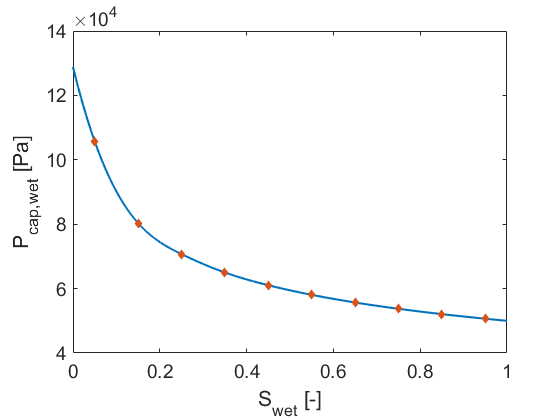}
    \caption{Capillary pressure curve with data samples and cubic spline interpolation. }
    \label{fig:capCurve}
\end{figure}

\begin{table}[htbp] 
\caption{Capillary curve data table} \label{tab:capData}
\centering 
\begin{tabularx}{.2\textwidth}{Y Y}

\toprule
{$S_{wet}$}                 & {$p_{cap,wet}$ } \\
\midrule
0.05        &       1.057e5 \\
0.15        &       8.034e4 \\
0.25        &       7.071e4 \\
0.35        &       6.500e4 \\
0.45        &       6.104e4 \\
0.55        &       5.806e4 \\
0.65        &       5.568e4 \\
0.75        &       5.372e4 \\
0.85        &       5.207e4 \\
0.95        &       5.064e4  \\
\bottomrule
\end{tabularx}
\end{table}

\bibliography{journalAbbreviationExtended,IHUWA} 


\end{document}